\documentclass[a4paper, 12pt]{article}
\usepackage[a4paper, total={6.5in, 9.5in}]{geometry}
\usepackage{graphicx} 
\usepackage{subfig}
\usepackage{multirow}%
\usepackage{amsmath,amssymb,amsfonts}%
\usepackage{amsthm}%
\usepackage{mathrsfs}%
\usepackage[title]{appendix}%
\usepackage{xcolor}%
\usepackage{textcomp}%
\usepackage{manyfoot}%
\usepackage{booktabs}%
\usepackage[ruled, linesnumbered]{algorithm2e}
\usepackage{listings}%
\usepackage{nicematrix}
\usepackage{enumitem}

\usepackage{pifont}

\usepackage{graphicx}
\usepackage{hyperref}
\hypersetup{
    hidelinks
}
\usepackage{float}
\usepackage{mathtools}
\usepackage{cleveref}

\newcommand{\NN}{\mathbb N}

\newcommand{\QQ}{\mathbb Q}

\newcommand{\ZZ}{\mathbb Z}

\DeclareUnicodeCharacter{2212}{-}

\theoremstyle{plain}
\newtheorem{theorem}{Theorem}[section]
\newtheorem*{theorem*}{Theorem}
\newtheorem{lemma}[theorem]{Lemma}
\newtheorem{proposition}[theorem]{Proposition}
\newtheorem{corollary}[theorem]{Corollary}
\newtheorem{conjecture}[theorem]{Conjecture}

\theoremstyle{definition}
\newtheorem{definition}[theorem]{Definition}
\newtheorem{example}[theorem]{Example}

\theoremstyle{remark}
\newtheorem{remark}[theorem]{Remark}

\usepackage{tikz}
\usetikzlibrary{shapes.geometric}
\usepackage{wrapfig}
\pgfdeclarelayer{nodelayer}
\pgfdeclarelayer{edgelayer}
\pgfsetlayers{nodelayer, edgelayer, main}

\title{
Distance Reduction in Bouquet Decompositions\\
and Toric Ideals of Graphs}

\author{Oliver Clarke, Dimitra Kosta and Alexander Milner}

\date{}

\begin{document}

\maketitle

\begin{abstract}
    The distance-reduction property for a generating set, i.e., a Markov basis, of a toric ideal is a condition that ensures tight connectivity of its fibres. In this paper, we study the distance-reduction property for toric ideals of graphs and move on to explore the relationship between the distance-reduction property and the bouquet structure of homogeneous toric ideals, which includes the class of toric ideals of graphs. For toric ideals of graphs which are complete intersection, we show that the minimal Markov bases are distance-reducing if and only if they distance-reduce the circuits of the ideal. We then consider how the distance-reduction properties interact with the bouquet structure of the toric ideal. Bouquets are a combinatorial  structure that capture the essential combinatorial information of the toric ideal.  Under the condition of homogeneity, we show that, for toric ideals with the same bouquet structure and signature, the distance-reduction properties are preserved. For homogeneous toric ideals whose bouquet matrix is a monomial curve in $\mathbb{A}^3$, we give necessary and sufficient conditions for when the minimal Markov bases are distance-reducing.

\end{abstract}

\tableofcontents





\section{Introduction}
\label{sec: intro}
 
A fundamental problem in Algebraic Statistics is the construction of Markov bases for toric ideals. 
Given an integer matrix $A$, a \emph{Markov basis} corresponds to a minimal generating set of the toric ideal $I_A$, and its elements can be considered as moves of a Markov chain that connects any two points in the same fiber of $A$ (see \cite{diaconis1998algebraic}).
The application of this Markov chain is in sampling from the space of so-called contingency tables, which may be used in Fisher's exact test in statistical significance testing.
In practice, we not only want connectivity in the fibers but a quick convergence rate of the Markov chain as well, i.e., a low mixing time.


This is the idea behind the \emph{distance-reduction property}, introduced by Aoki
and Takemura \cite{aoki2005distance, aoki2012markovbook}, which requires a Markov basis $M$ to reduce a certain metric distance between points in a fiber. In this case, the diameter of the fibers can be bounded from above in terms of their fiber values \cite[Section 3.1]{aoki2005distance} and these diameters can, in turn, be used to bound the time for Markov chains to reach their stationary distribution (see \cite{diaconis1998algebraic} or \cite{DSC_1998}). In particular, we will focus on the metric given by the $1$-norm on the fibers, which is not only easy to compute, but also has nice properties when working with homogeneous toric ideals.

Understanding which Markov bases are distance-reducing is a subtle question. 
By \cite{CK_2024}, for a general integer matrix $A$, 
checking whether a Markov basis $M$ is distance-reducing requires verifying the condition against every element of the Graver basis ${\rm Gr}(A)$, the set of primitive lattice vectors in $\ker_{\ZZ}(A)$.
It follows immediately that the Graver basis is always distance-reducing, but it is typically orders of magnitude larger than any given Markov basis. 
So, a natural goal is to identify structured families of toric ideals for which distance reduction can be certified by checking an easily accessible set of elements.
For instance, in this paper, we will show that the \emph{circuits} of $A$, which are the elements of $I_A$ of inclusion-wise minimal support, are sufficient for testing distance reduction in the case of complete intersection toric ideals of graphs.

\medskip
In this paper, we study the distance-reduction property from two perspectives. Firstly, we consider \emph{toric ideals of graphs} and leverage their combinatorial characterisation to study the distance-reduction property of their Markov bases. Secondly, we consider how the \emph{bouquet} structure of a homogeneous toric ideal affects the distance-reduction property of its Markov bases. Homogeneity is a unifying feature of both perspectives, as toric ideals of graphs are homogeneous.

\medskip \noindent
\textbf{Toric ideals of graphs.} Toric ideals of graphs, their generating sets as well as several binomial sets of interest have been studied in a series of papers; the circuits of toric ideals associated to graphs were characterised by Villarreal in \cite{Vil_1995},  the Graver basis of the toric ideals of a graph was first investigated by Ohsugi and Hibi in \cite{ohsugi-hibi-quad-binoms}, while the explicit structure of its elements was later established by Reyes, Tatakis and Thoma in \cite{Reyes2012minimal}. Moreover, in \cite{DeLoera_etal} De Loera, Sturmfels and Thomas determined  the universal Gr\"obner basis for toric ideals of graphs with at most eight vertices, while Tatakis and Thoma gave a graph theoretical characterisation of the elements of the universal Gröbner basis of the toric ideal of any graph in \cite{TATAKIS_Thoma_UGrobner}.   

The incidence matrix $A_G$ of a simple graph $G$ provides a natural class of $0/1$-matrices, each of whose columns contains exactly two ones. 
When it comes to distance-reduction, toric ideals of graphs are the natural next class to study after the trivial case with exactly one 1 in each column, and before the full generality of matrices with at most three ones per column, where any bouquet structure can appear, see \cite{petrovic2019hypergraph}. 
Note that, since every column has exactly two ones, the toric ideal of $A_G$ is homogeneous.
We note that, for a homogeneous ideal, by Proposition~\ref{prop:drredef}, the distance-reduction condition simplifies to the following: an element $z$ is distance-reduced by $\hat{z}$ 
if and only if $\hat{z}^+ \leq z^+$ and ${\rm Supp}(\hat{z}^-) \cap {\rm Supp}(z^-) \neq \emptyset$ up to the signs of $\hat z$ and $z$. We exploit this simplification throughout.

In the setting of toric ideals of graphs, the Graver basis is in bijection with the primitive even closed walks of $G$ by \cite{ohsugi-hibi-quad-binoms} and \cite{sturmfels1996grobner}, and the circuits correspond precisely to the subset of even cycles and pairs of odd cycles either intersecting in exactly one vertex or being vertex-disjoint and joined by a path,  see \cite[Proposition 4.2]{Vil_1995}. 
Our first main result is about complete intersection toric ideals of graphs for which we prove the following. 

\begin{theorem*}[Theorem~\ref{thm:CImain}]
    Let $G$ be a simple connected graph such that $I_G$ is a complete intersection,
    and let $M$ be a minimal Markov basis for $I_G$. Then $M$ is distance-reducing if
    and only if $M$ distance-reduces the circuits of $A_G$.
\end{theorem*}

The statement of Theorem \ref{thm:CImain} is motivated by \cite[Corollary 5.2]{CK_2024} where it is shown that, for $A \in \ZZ^{1 \times n}$, which gives rise to a monomial curve, such that $I_A$ is a complete intersection toric ideal, a minimal Markov basis is distance-reducing if and only if it distance-reduces the circuits of $A$. The proof of this result involves exploiting the specific form of minimal Markov bases of complete intersection monomial curves. Analogously, we prove Theorem \ref{thm:CImain} using a result by Tatakis and Thoma, restated in Theorem \ref{thm:tatakisCI}, about the structure of minimal Markov bases of complete intersection toric ideals coming from graphs.
However, we suspect that the complete intersection hypothesis can be dropped from this theorem, see Conjecture~\ref{conj: distance reduction for graph ideals by circuits}, which we have computationally verified for all graphs on at most $7$ vertices and all graphs with $8$ vertices and at most $16$ edges.

We note that if we try to extend Conjecture~\ref{conj: distance reduction for graph ideals by circuits} to all toric ideals arising from $0/1$-matrices, then the statement is immediately false. In fact, the following example shows that there exists a matrix $A$ that contains two $1$'s in each column except one, which contains three $1$'s, and a Markov basis $M$ for $A$ such that $M$ distance-reduces the circuits of $A$ but not the Graver basis.

\begin{example}
Consider the matrix 
\[
A = 
    \begin{bmatrix}
      1&1&0&0&0&0&0&0&0&0&0&0\\
      1&0&1&1&0&0&0&0&0&0&0&0\\
      1&0&0&0&1&0&0&0&0&0&0&0\\
      0&0&0&0&0&1&1&1&1&0&0&0\\
      0&0&1&0&0&1&0&0&0&1&1&0\\
      0&1&0&0&1&0&1&0&0&0&0&1\\
      0&0&0&0&0&0&0&1&0&1&0&1\\
      0&0&0&1&0&0&0&0&1&0&1&0 
    \end{bmatrix},
\]
which is \textit{almost} the incidence matrix of a graph except that the first column contains three $1$s.
Using the \texttt{AllMarkovBases} package \cite{CM_2025} for \textit{Macaulay2} \cite{M2}, we observe that the toric ideal $I_A$ has $27$ distinct  Markov bases. The Graver basis ${\rm Gr}(A)$ consists of $87$ elements and the circuits are a subset of $38$ elements. Let us fix the Markov basis
\[
M = 
\begin{bmatrix}
      0&0&0&0&0&0&0&1&-1&-1&1&0\\
      0&0&0&0&0&1&-1&0&0&-1&0&1\\
      0&0&1&-1&0&-1&0&0&1&0&0&0\\
      2&-2&-2&0&-2&-1&4&0&-3&0&3&0\\
      2&-2&-2&0&-2&0&3&-3&0&2&0&1\\
      2&-2&-2&0&-2&0&3&-1&-2&0&2&1\\
      2&-2&-1&-1&-2&0&2&-2&0&0&1&2\\
      2&-2&0&-2&-2&0&1&-3&2&0&0&3\\
      2&-2&0&-2&-2&0&1&-2&1&-1&1&3\\
      2&-2&0&-2&-2&0&1&-1&0&-2&2&3\\
      2&-2&1&-3&-2&0&0&0&0&-4&3&4
\end{bmatrix}.
\]
The Markov basis $M$ reduces the distance of all circuits of $A$ but it fails to reduce the distance of the Graver basis element
$
z = \begin{pmatrix}2& -2& -2& 0& -2& 0& 3& -2& -1& 1& 1& 1 \end{pmatrix}
$.
To see this, note that the only elements of $M$ that are applicable to $z$ are
\[
u = \begin{pmatrix}0&0&0&0&0&1&-1&0&0&-1&0&1\end{pmatrix}\,,
\
v = \begin{pmatrix}0&0&1&-1&0&-1&0&0&1&0&0&0\end{pmatrix}.
\]
Observe that $u^- \le z^+$ and $v^+ \le z^-$, but $||z+u|| = ||z||$ and $||z+v|| = ||z||$. Thus $M$ does not distance-reduce $z$, hence $M$ is not distance-reducing. Thus Conjecture~\ref{conj: distance reduction for graph ideals by circuits} is false if we replace the incidence matrix $A$ with an arbitrary $0/1$-matrix.
\end{example}

\medskip \noindent 
\textbf{Bouquet decompositions.} The bouquet structure of a toric ideal, introduced by Petrovi\'c, Thoma, and Vladoiu in \cite{petrovic2018bouquet}, partitions the columns of $A$ into subsets called \emph{bouquets}.
Each of these bouquets is said to be free, non-mixed, or mixed according to its \emph{signature}. 
This decomposition captures a great deal of the essential combinatorial structure of $I_A$.
For instance, it is known that two matrices that share the same bouquet ideal and signature have isomorphic sets of circuits, Graver bases, indispensable elements and even minimal Markov bases, with an explicit bijection between them (see \cite{petrovic2018bouquet, kosta2023strongly}).
We investigate the natural question of whether the distance-reduction property is also an invariant of the bouquet ideal and signature.

We show that the answer is no in general. See Example~\ref{ex: same bouquet diff DR}, which exhibits two matrices with the same bouquet ideal and signature, each of which has two minimal Markov bases. But, for one of these matrices both minimal Markov bases are distance-reducing while for the other matrix only one is distance-reducing. This difference arises from the observation that the bijection does not preserve the $1$-norms of elements in fibers. 
However, we show that imposing homogeneity on both toric ideals completely removes this obstruction.

\begin{theorem*}[Theorem~\ref{thm:globaldr}]
    Let $I_{A_1},I_{A_2}$ be homogeneous $T_{\omega}$-robust ideals for some integer matrix $T$ and signature $\omega$. Then, $I_{A_1}$ and $I_{A_2}$ have the same distance-reducing properties: for $\hat z,z \in \ker_\ZZ(A_1)$, the element $z$ is distance-reduced by $\hat z$ if and only if $\phi(z)$ is distance-reduced by $\phi(\hat z)$, where $\phi$ is a certain bijection between the kernels.
\end{theorem*}

For this result, we think of the distance-reduction property in two parts: 
there is a divisibility condition on positive parts and an overlapping condition on supports of negative parts.
The proof of this theorem passes through two lemmas on the \emph{poset of signatures} in which we track how
these two conditions change as we move through this poset. 
In particular, we show that the divisibility condition is preserved moving up the poset, 
while the support-overlap condition is preserved by moving down. 
Therefore, both these conditions are preserved when we consider two toric ideals corresponding to the same element of the poset. 
It is worth noting that even in the non-homogeneous case, the divisibility condition is still preserved moving up the poset. However, homogeneity allows us to work with this support-overlap condition instead of the more complicated 1-norm bound condition, which is not preserved moving down the poset in the non-homogeneous case.

We also give an intrinsic characterisation of homogeneity in terms of the bouquet data. 
We show that an ideal $I_A$ is homogeneous if and only if the vector $\mathbf{\bar{c}} = (\mathbf{\bar{c}}_1, \dots, \mathbf{\bar{c}}_s)$ of column-sums of the bouquet-index-encoding vectors lies in the row span of the bouquet matrix $A_B$. See Theorem~\ref{thm:homrow}. This gives an efficient way to construct examples of homogeneous $T_{\omega}$-robust ideals and demonstrates that homogeneity depends only on the data of the bouquet matrix and the bouquet-index-encoding vectors.

\medskip \noindent
\textbf{Classification over $A_3$.} Finally, we turn to a complete classification of distance reduction for toric ideals whose bouquet matrix is a monomial curve over $\mathbb{A}^3$ given by a matrix $T = \begin{pmatrix} n_1 & n_2 & n_3 \end{pmatrix}$. 
In this setting, the ideal $I_T$ and hence the original ideal, have exactly three circuits, whose structure is completely explicit and we prove the following.
 
\begin{theorem*}[Theorem~\ref{thm:A3class}]
	Let $T = \begin{pmatrix} n_1 & n_2 & n_3 \end{pmatrix} \in \ZZ_{>0}^3$, $\omega \subseteq [3]$, and let $A$ be a homogeneous $T_\omega$-robust matrix with minimal Markov basis $M$. Then $M$ fails to be distance-reducing if and only if all three of the following conditions hold: $\omega = [3]$, the monomial curve $T$ is a complete intersection, and $M$ contains two elements that are both circuits. In this case, $M$ distance-reduces every element of ${\rm Gr}(A)$ except the unique circuit of $A$ not contained in $M$.
\end{theorem*}

\medskip \noindent
\textbf{Outline.} In Section~\ref{sec: background}, we recall the necessary background on toric ideals, distance reduction, and the homogeneity condition, including Theorem~\ref{thm:DAexact} that characterises distance-irreducible elements using the fiber graph. 
In Section~\ref{sec: graphs}, we study distance reduction for toric ideals of graphs, which concludes with a proof of the complete intersection result Theorem~\ref{thm:CImain}. 
In Section~\ref{sec:Bouquets}, we develop the theory of distance reduction in terms of the bouquet structure. We establish the invariance theorem under homogeneity, and provide the bouquet-theoretic
characterisation of homogeneity. At the end of this section, in Section~\ref{sec: distance reduction A3}, we study toric ideals whose bouquet matrix $T \in \mathbb{Z}_{>0}^3$ is a monomial curve.





\section{Background}\label{sec: background}

We write $\NN = \{0,1,2,\dots\}$ for the set of non-negative integers and $[n] := \{1,2, \dots, n\}$. Given an integer vector $z \in \ZZ^n$, we define $z^+, z^- \in \NN^n$ to be the unique non-negative integer vectors such that $z = z^+ - z^-$. Throughout, we use the $1$-norm on $\ZZ^n$, which is written as $\lVert \cdot \rVert$. For any vector $z \in \NN^n$, we denote its support by $\operatorname{Supp}(z) = \{i \in [n] : z_i \neq 0\}$. We use $\le$ to denote the component-wise partial order on $\NN^n$ that extends the natural order on $\NN$ explicitly given by
\[
\text{if }u,v \in \NN^n \text{ then } u \le v \iff u_i \le v_i \text{ for all } i \in [n].
\]
Equivalently, we have $u \le v$ if $v-u \in \NN^n$.
Given a polynomial ring $R = k[x_1, \dots, x_n]$ and a vector $u \in \NN^n$, we use multi-index notation for monomials $x^u := x_1^{u_1} \dots x_n^{u_n}$.

\subsection{Toric ideals}

Let $A \in \ZZ^{d \times n}$ be an integer matrix and $R = k[x_1, \dots, x_n]$ a polynomial ring over a field $k$. Assume that $\ker_\ZZ(A) \cap \NN^n = \{0\}$. For each $z \in \ker_\ZZ(A)$, we write $\beta(z) = x^{z^{+}} - x^{z^-}$ for the binomial associated to $z$. We define the toric ideal $I_A = \langle  \beta(z) : z \in \ker_\ZZ(A) \rangle$. 
We say that $I_A$ (or $A$) is \textit{homogeneous} if the vector $\mathbf{1}_n = (1,1, \dots, 1) \in \ZZ^n$ appears in the row-span of $A$.

The \textit{affine semigroup of $A$} is $\NN A := \{Au \in \ZZ^d \,|\, u \in \NN^n\}$, which is the semigroup generated by the columns of $A$. 
Given a monomial $x^u \in R$ for some $u \in \NN^n$, its \textit{$A$-degree} is defined to be the vector $\deg_A(x^u) = Au \in \NN A$. 
For each $b \in \NN A$, the $b$-fiber of $A$ is given by $\mathcal F_b := \{u \in \NN^n \,|\, Au = b\}$.

{\em Circuits} are irreducible binomials of a toric ideal $I_A$ with minimal support. In vector notation, a vector ${\bf u}\in \ker_{\ZZ}(A)$ is called a circuit of the matrix $A$ if $\operatorname{Supp}({\bf u})$ is minimal and the components of ${\bf u}$ are relatively prime. We write $C(A) \subseteq \ker_\ZZ(A)$ for the set of circuits of $A$.


A \textit{Markov basis} $M \subseteq \ker_\ZZ(A)$ for $A$ (or for $I_A$) is a collection of vectors such that the set $\{\beta(z) : z \in M\}$ is a generating set of $I_A$. It is \textit{minimal} if this generating set is minimal and we ignore the signs of the vectors. The intersection of all Markov bases for $A$ is the \textit{indispensable set}, denoted by $S(A)$, and the union of Markov bases is called the \textit{universal Markov basis}, and is denoted by $U(A)$. For ease of notation, a set of vectors in $\ker_{\ZZ}(A)$, such as a Markov basis, is denoted by a matrix whose rows are the set of vectors. 
We recall the following forms of vector decomposition for exponents in toric ideals.

\begin{definition}\label{def: conformal and semi-conformal decomposition}
    Fix $A$ as above and let $z, u, v \in \ker_\ZZ(A)$ be vectors such that $z = u+v$. We say that the decomposition $z = u+v$ is:
    \begin{itemize}
        \item \textit{proper} if neither $u$ nor $v$ is the zero vector,
        \item \textit{conformal} if $ z^+=u^++ v^+ \text{ and } z^-=u^-+v^-$,
        \item \textit{semi-conformal} if $u_i > 0 \implies v_i \geq 0$ and $v_i < 0 \implies u_i \leq 0$.
    \end{itemize}
    We define the \textit{Graver basis} ${\rm Gr}(A) \subseteq \ker_\ZZ(A)$ as the set of elements that do not admit a proper conformal decomposition. 
    A binomial $x^u-x^v \in I_A$ is called \textit{primitive} if there is no other binomial $x^{u'}-x^{v'} \in I_A$ such that $x^{u'}|x^u$ and $x^{v'}|x^v$.
\end{definition}
Let $\operatorname{Gr}(I_A) := \{\beta(z) \,|\, z \in \operatorname{Gr}(A)\}$. The set of
primitive binomials is finite and coincides with $\operatorname{Gr}(I_A)$. Additionally, the circuits of $A$ are contained within the Graver basis of $A$. We reference \cite[Chapter 4]{sturmfels1996grobner} for both of these results.

Semi-conformal decompositions allow us to give an equivalent characterisation of the indispensable set.

\begin{proposition}[{\cite[Proposition~1.1]{charalambous2014markov}}]\label{prop: no semi-conformal decomp iff indispensable}
    The set $S(A)$ of indispensable elements is equal to the subset of $\ker_\ZZ(A)$ that do not admit a proper semi-conformal decomposition.
\end{proposition}

\subsection{Distance reduction}

The theory of distance reduction, introduced in \cite{aoki2005distance}, is about studying $1$-norm distance between points in fibers of $A$. In this section we recall the basic definitions and prove some small results that will be used throughout the paper.

\vspace{1.5mm}

\begin{definition}\label{def:drbouquet}
    For $\hat z,z \in \ker_{\ZZ}(A)$, we say that \textit{$z$ is distance-reduced by $\hat z$} if
    \begin{enumerate}
        \item \label{item:drbouquet1} $\hat z^+ \le z^+$
        \item \label{item:drbouquet2} $\Vert z - \hat z \Vert < \Vert z \Vert$
    \end{enumerate}
    up to the signs of $\hat z$ and $z$.
    Explicitly, we require at least one of the following to hold:
    \begin{itemize}
        \item $\hat z^+ \le z^+$ and $\Vert z - \hat z \Vert < \Vert z \Vert$,
        \item $\hat z^+ \le z^-$ and $\Vert z + \hat z \Vert < \Vert z \Vert$,
        \item $\hat z^- \le z^+$ and $\Vert z + \hat z \Vert < \Vert z \Vert$,
        \item $\hat z^- \le z^-$ and $\Vert z - \hat z \Vert < \Vert z \Vert$.
    \end{itemize}
    We say that an element $z \in \ker_\ZZ(A)$ is \textit{distance-irreducible} if the only element of $\ker_\ZZ(A)$ that distance-reduces $z$ is itself. We write $D(A)$ for the set of distance-irreducible elements.
\end{definition}



The intuition behind Definition \ref{def:drbouquet} comes from considering the fiber that $z^+$ and $z^-$ lie in. Let $b:=A z^+=Az^-$ so $z^+$ and $z^-$ lie in the fiber $\mathcal{F}_b$. Then, condition (\ref{item:drbouquet1}) of Definition \ref{def:drbouquet} tells us that $\hat z$ is ``applicable" to $z$ i.e., applying the move $\hat z$ to $z^+$ keeps us in the fiber $\mathcal{F}_b$. We can see this by noting that $z^+-\hat z=(z^+-\hat z^+)+\hat z^- \in \NN^n$ since $\hat z^- \in \NN^n$ and, by condition (\ref{item:drbouquet1}) of Definition \ref{def:drbouquet} , $z^+-\hat z^+ \in \NN^n$. Furthermore, $A(z^+-\hat z)=A z^+=b$ since $\hat z \in \ker_\ZZ(A)$ so $\hat z^+-\hat z$ lies in the fiber $\mathcal{F}_b$. Then, rewriting condition (\ref{item:drbouquet2}) of Definition \ref{def:drbouquet}, $\lVert (z^+-\hat z)-z^-\rVert<\lVert z^+ - z^-\rVert$, this is telling us that after applying the move $\hat z$ to $z^+$, our new position in the fiber, which is $z^+-\hat z$, is closer to $z^-$ than $z^+$ was. Thus, $\hat z$ allows us to actually ``reduce the distance" between $z^+$ and $z^-$.

We will show in Proposition~\ref{prop:drredef} that the homogeneity condition on the toric ideal $I_A$ allows us to simplify part \ref{item:drbouquet2} of Definition \ref{def:drbouquet}. This all comes down to Lemma \ref{lem:hompm} which tells us that every element in the same fiber has the same $1$-norm. Thus, when we move from $z^+$ to $z^+-\hat z$ in our fiber, since the supports of $z^+$ and $z^-$ are completely disjoint, all we need to decrease the distance to $z^-$ is to have some overlap between the supports of $\hat z$ and $z^-$. However, by condition (\ref{item:drbouquet1}) of Definition \ref{def:drbouquet} , $\hat z^+ \le z^+$ so, in particular, $\operatorname{Supp}(\hat z) \subseteq \operatorname{Supp}(z^+)$. But again, the supports of $z^+$ and $z^-$ are completely disjoint so the only way for the support of $\hat z$ to overlap with $z^-$ is for the support of $\hat z^-$ to overlap with $z^-$. Proposition \ref{prop:drredef} shows that, alongside condition (\ref{item:drbouquet1}) of Definition \ref{def:drbouquet}, this is a necessary and sufficient condition for distance reduction in the homogeneous setting.

\begin{lemma} \label{lem:hompm}
    If $A$ is homogeneous and $z \in \ker_\ZZ(A)$, then we have $\lVert z^+ \rVert= \lVert z^- \rVert$.
\end{lemma}

\begin{proof}
    Since $A$ is homogeneous, there exists $\theta \in \QQ^d$ such that $\theta A = \mathbf{1}_n$. Then, 
    \begin{align*}
        \lVert z^+ \rVert - \lVert z^- \rVert&= \left(\sum_{k=1}^n |z^+_k| \right)-\left(\sum_{k=1}^n |z^-_k| \right)=\left(\sum_{k=1}^n z^+_k \right)-\left(\sum_{k=1}^n z^-_k \right) \\
        &=\left(\sum_{\substack{k=1 \\ z_k>0}}^n z_k \right)-\left(\sum_{\substack{k=1 \\ z_k<0}}^n -z_k \right)=\sum_{k=1}^n z_k=\mathbf{1}_n \cdot z 
        =(\theta A) \cdot z = \theta (Az)=0.
    \end{align*}
\end{proof}

\begin{proposition} \label{prop:drredef}
    Let $I_A$ be a homogeneous toric ideal and let $\hat z,z \in \ker_\ZZ(A)$. Then, $z$ is distance-reduced by $\hat z$ if and only if $\hat z^+ \le z^+$ and $\operatorname{Supp}(\hat z^-) \cap \operatorname{Supp}(z^-)$ is non-empty up to the signs of $\hat z$ and $z$.
\end{proposition}

\begin{proof}
    For the ($\Rightarrow$) direction, since $z$ is distance-reduced by $\hat z$, by condition (\ref{item:drbouquet1}) of Definition~\ref{def:drbouquet}, we have $\hat z^+ \le z^+$ up to switching the signs of $\hat z$ and $z$. We now suppose $\operatorname{Supp}(\hat z^{-}) \cap \operatorname{Supp}(z^{-})$ is empty and aim to find a contradiction by showing that condition (\ref{item:drbouquet2}) of Definition \ref{def:drbouquet} is not fulfilled. We first note that $\operatorname{Supp}(z^{-}) \cap \operatorname{Supp}(z^{+})$ is empty and, since $\hat z^{+} \le z^{+}$, then $\operatorname{Supp}(z^{-}) \cap \operatorname{Supp}(\hat z^{+})$ is also empty. Thus, the support of $z^{-}$ is distinct from the supports of $z^{+}$, $\hat z^{+}$ and $\hat z^{-}$ so,
\begin{equation}\label{eq:redef1}
 \Vert z - \hat z \Vert =\Vert  z^{+} - z^{-} -  \hat z^{+} + \hat z^{-} \rVert
        =\lVert  z^{+} -  \hat z^{+} + \hat z^{-}\rVert + \lVert - z^{-}\rVert.
\end{equation}
Note that $\hat z^{-} \in \NN^n$ and, since $\hat z^{+} \le z^{+}$ then $z^{+}-\hat z^{+} \in \NN^n$ meaning 
\begin{align}\label{eq:redef2}
 \lVert  z^{+} -  \hat z^{+} + \hat z^{-}\rVert&=\sum_{k=1}^n |(z^{+} -  \hat z^{+})_k + \hat z_k^{-}|=\sum_{k=1}^n |(z^{+} -  \hat z^{+})_k | + \sum_{k=1}^n|\hat z_k^{-}|\\ \nonumber
 &=\lVert  z^{+} -  \hat z^{+} \rVert + \lVert \hat z^{-}\rVert.
\end{align}
    Now, $\hat z^{+} \le z^{+}$ also implies that $\lVert z^{+} - \hat z^{+} \rVert = \lVert z^{+} \rVert - \lVert \hat z^{+} \rVert$ since $z^+,\hat z^+ \in \NN^n$ so, combining this with Equations (\ref{eq:redef1}) and (\ref{eq:redef2}), we have
    \begin{align*}
        \Vert z - \hat z \Vert=\lVert  z^{+} \rVert -  \lVert \hat z^{+} \rVert +  \lVert\hat z^{-}\rVert + \lVert - z^{-}\rVert.
    \end{align*}
    Since $I_A$ is homogeneous, by Lemma~\ref{lem:hompm}, we have $\lVert \hat z^+ \rVert = \lVert \hat z^- \rVert$ and we are left with
    \begin{equation} \label{eq:redef3}
        \Vert z - \hat z \Vert=\lVert  z^{+} \rVert+\lVert - z^{-}\rVert=\lVert  z^{+} - z^{-}\rVert=\lVert z \rVert
    \end{equation}
    since $\operatorname{Supp}(z^{-}) \cap \operatorname{Supp}(z^{+})$ is empty. By Equation~(\ref{eq:redef3}), we see that condition (\ref{item:drbouquet2}) of Definition~\ref{def:drbouquet} is not fulfilled which was the contradiction we were looking for.

    \smallskip
    For the ($\Leftarrow$) direction, assume $\hat z^+ \le z^+$ and $\operatorname{Supp}(\hat z^{-}) \cap \operatorname{Supp}(z^{-})$ is non-empty meaning we have $i \in \operatorname{Supp}(\hat z^{-}) \cap \operatorname{Supp}(z^{-})$ for some $i \in [n]$. This means $z_i^->0$ and $\hat z^-_i>0$ so $|\hat z_i^- - z_i^- |<|\hat z_i^-| + |z_i^-|$.
    \begin{align} \label{eq:redef3.5}
        \lVert \hat z^{-} - z^{-} \rVert&=\sum_{k=1}^n |\hat z_k^{-}-z_k^{-}|=|\hat z^{-}_i- z^{-}_i|+\sum_{\substack{k =1 \\ k \neq i}}^n |\hat z^{-}_k-z^{-}_k| \\ \nonumber
        &<|\hat z^{-}_i|+|-z^{-}_i|+\sum_{\substack{k =1 \\ k \neq i}}^n |\hat z^{-}_k-z^{-}_k| \\ \nonumber
        &\le |\hat z^{-}_i|+|-z^{-}_i|+\sum_{\substack{k =1 \\ k \neq i}}^n \left(|\hat z^{-}_k|+|-z^{-}_k|\right) \\ \nonumber
        &=\sum_{k =1}^n \left(|\hat z^{-}_k|+|-z^{-}_k|\right)=\sum_{k =1}^n |\hat z^{-}_k|+\sum_{k =1}^n|-z^{-}_k|\\ \nonumber
        &=\lVert \hat z^{-} \rVert + \lVert -z^{-} \rVert.
    \end{align}
    Finally, since $\lVert \cdot \rVert$ is a norm, by the triangle inequality, we have
    \begin{equation}\label{eq:redef4}
         \lVert z - \hat z \rVert=\Vert  z^{+}- z^{-} -  \hat z^{+} + \hat z^{-} \rVert \le \lVert  z^{+} -  \hat z^{+} \rVert + \lVert \hat z^{-} - z^{-} \rVert.
    \end{equation}
   As before, $\hat z^{+} \le z^{+}$ implies that $\lVert z^{+} - \hat z^{+} \rVert = \lVert z^{+} \rVert - \lVert \hat z^{+} \rVert$ and, combining this and Equation~(\ref{eq:redef3.5}) with Equation~(\ref{eq:redef4}), we have
   \begin{equation} \label{eq:redef5}
       \lVert z - \hat z \rVert<  \lVert  z^{+} \rVert - \lVert  \hat z^{+} \rVert + \lVert \hat z^{-} \rVert +\lVert - z^{-} \rVert
   \end{equation}
   By Lemma~\ref{lem:hompm} again, we have $\lVert \hat z^+ \rVert = \lVert \hat z^- \rVert$ and so we are left with
     \begin{equation}  \label{eq:redef6}
       \lVert z - \hat z \rVert<  \lVert  z^{+} \rVert + \lVert -z^{-} \rVert=\lVert  z^{+} - z^{-} \rVert=\lVert z \rVert
   \end{equation}
    since $\operatorname{Supp}(z^{-}) \cap \operatorname{Supp}(z^{+})$ is empty. By assumption $\hat z^+ \le z^+$, so condition (\ref{item:drbouquet1}) of Definition~\ref{def:drbouquet} holds and, by Equation~(\ref{eq:redef6}), part \ref{item:drbouquet2} also holds. Thus, $z$ is distance-reduced by $\hat z$.
\end{proof}

\subsection{Homogeneity and distance reduction}


Proposition \ref{prop:drredef} demonstrated how the homogeneity condition somewhat simplified condition (\ref{item:drbouquet2}) of Definition \ref{def:drbouquet}. We now present a small example of how this simplification manifests when dealing with distance-irreducible elements. Recall from Definition \ref{def:drbouquet} that distance-irreducible elements are those elements which are only distance-reduced by themselves.

For non-homogeneous toric ideals, the distance-irreducible elements are not always a subset of the universal Markov basis. For instance, in \cite{CK_2024}, it is shown that if $A=\begin{bmatrix} 8&31&33&53\end{bmatrix}$, then  $D(A) \not\subseteq U(A)$. However, as we will now see, if we restrict ourselves to homogeneous toric ideals, we always have $D(A) \subseteq U(A)$. This result will follow from Theorem~\ref{thm:DAexact}, but to prove it we require the following important definitions for graphs associated to fibers. 

\begin{definition}[{\cite[Definition~2.1]{CKT_2007}}] \label{def:fiberGraph}
    For each $b \in \NN A$, define the toric ideal
$$I_{A,b} := \langle x^u-x^v \,|\,  \deg_A(x^u)=\deg_A(x^v), \  b - \deg_A(x^v) \in \NN A \setminus \{0\} \rangle.$$
We construct the graph $G_b$ with vertex set $\mathcal F_b$ and edge set 
$$E(G_b) = \{\{u,v\} \,|\, x^u-x^v \in I_{A,b}\}.$$
\end{definition}

\begin{definition}[{\cite[Section~2.1]{CM_2025}}] \label{def:3difgra}
    We also define the graph $G'_b$ on the vertex set $\mathcal{F}_b$ whose edges are  $$E(G'_b)=\{\{u,v\} \,|\, \operatorname{Supp}(u) \cap \operatorname{Supp}(v) \neq \emptyset\}.$$
\end{definition}

\begin{proposition} \label{prop:3graphs}
    The graphs $G_b$ and $G'_b$ have the same connected components.
\end{proposition}

The proof of this result is a little technical and unrelated to the current narrative, so it appears in the appendix as Proposition~\ref{prop:samegraphccs}.

\begin{theorem}\label{thm:DAexact}
    For $A$ homogeneous and $z \in \ker_\ZZ(A)$ with $Az^+=Az^-=b$, then $z \in D(A)$ if and only if $z^+,z^-$ are isolated vertices of $G_b$.
\end{theorem}

\begin{proof}
     $(\Longrightarrow)$ Suppose $z \in D(A)$. Aiming for a contradiction, suppose $z^+$ and $z^-$ are not both isolated vertices in $G_b$. Up to a change of sign of $z$, we can take $z^-$ not to be an isolated vertex in $G_b$. By Proposition \ref{prop:3graphs}, $z^-$ is not an isolated vertex in $G'_b$ either so we can find some $m \in \mathcal{F}_b$ such that $m \neq z^-$ and $\{m,z^-\} \in E(G'_b)$. Now, let $\hat z:=z^+-m \in \ker_\ZZ(A)$ where we notice $\hat z \neq z$. Since $m \in \NN^n$ then $(-m)^+=0$, so $\hat z^+=(z^+-m)^+ \le (z^+)^+ + (-m)^+=z^+$. Also, since $\{m,z^-\} \in E(G'_b)$, we can find some $i \in \operatorname{Supp}(m) \cap \operatorname{Supp}(z^-)$. Since $i \in \operatorname{Supp}(z^-)$ and $\operatorname{Supp}(z^+) \cap \operatorname{Supp}(z^-)=\emptyset$ then $i \not\in \operatorname{Supp}(z^+)$ so $z^+_i=0$. Thus, $\hat z_i=z^+_i-m_i=-m_i<0$ since $i \in \operatorname{Supp}(m)$ so $i \in \operatorname{Supp}(\hat z^-)$. Putting this all together we have $\hat z^+ \le z^+$ and $i \in \operatorname{Supp}(\hat z^-) \cap \operatorname{Supp}(z^-)$ so, by Proposition \ref{prop:drredef}, $z$ is distance-reduced by $\hat z$ where $\hat z \neq z$. Thus, $z \not\in D(A)$ which is a contradiction.
    
    \smallskip
    $(\Longleftarrow)$ Suppose $z^+,z^-$ are isolated vertices of $G_b$. By Proposition \ref{prop:3graphs}, this implies that $z^+$ and $z^-$ are isolated vertices of $G'_b$ as well. Aiming for a contradiction, suppose $z \notin\ D(A)$ i.e., we have some $\hat z \in \ker_\ZZ(A)$ such that $\hat z$ distance-reduces $z$ and $\hat z \neq z$. By Proposition~\ref{prop:drredef}, up to the signs of $\hat z$ and $z$, we have $\hat z^+ \le z^+$ and $\operatorname{Supp}(\hat z^-) \cap \operatorname{Supp}(z^-)\neq \emptyset$. Since $\hat z^+ \le z^+$ then $z^+-\hat z^+ \in \NN^n$ so $z^+-\hat z=z^+-\hat z^++\hat z^- \in \NN^n$. In addition, since $\hat z \in \ker_\ZZ(A)$, $A(z^+-\hat z)=Az^+=b$ and thus $z^+-\hat z \in \mathcal{F}_b$. Since $\operatorname{Supp}(\hat z^-) \cap \operatorname{Supp}(z^-)$ is non-empty then we can take $i \in \operatorname{Supp}(\hat z^-) \cap \operatorname{Supp}(z^-)$. Then, since $i \in \operatorname{Supp}(\hat z^-)$, it follows that $i \not\in \operatorname{Supp}(\hat z^+)$ so $i \in \operatorname{Supp}(z^+-\hat z)$. Then, we have $i \in \operatorname{Supp}(z^+-\hat z) \cap \operatorname{Supp}(z^-)$ meaning $\{z^+-\hat z,z^-\} \in E(G'_b)$ which contradicts the fact that $z^-$ is an isolated vertex of $G'_b$. We have a contradiction so $z \in D(A)$.
    

    
    
\end{proof}

\begin{corollary}
    For $A$ homogeneous, we have
    $S(A) \subseteq D(A) \subseteq U(A) \subseteq \operatorname{Gr}(A)$.
\end{corollary}

\begin{proof}
    Take $z \in S(A)$ and let $b = Az^+=Az^-$. Then, by \cite[Corollary 2.10]{CKT_2007}, $G_b$ consists of two connected components, $\{z^+\}$ and $\{z^-\}$. Thus, $z^+,z^-$ are isolated vertices of $G_b$ so, by Theorem \ref{thm:DAexact}, $z \in D(A)$. So we have $S(A) \subseteq D(A)$.

    Now, take $z \in D(A)$. By Theorem \ref{thm:DAexact}, $z^+$ and $z^-$ are isolated vertices. By \cite[Theorem~2.6]{CKT_2007}, there is a minimal Markov basis containing $z$ and therefore $z \in U(A)$. And so we have shown that $D(A) \subseteq U(A)$.

    Finally, it is well known that $U(A) \subseteq \operatorname{Gr}(A)$, for example, see \cite{sturmfels1996grobner}.
\end{proof}

In Section \ref{sec:Bouquets}, we will see that, in the context of bouquet decompositions of toric ideals, similarly to above, the homogeneity condition guarantees that distance-reduction properties are conserved between toric ideals with the same bouquet ideal and signature whereas these are not guaranteed to hold in the non-homogeneous case.

\section{Toric ideals of graphs}\label{sec: graphs}

In this section we study distance reduction for toric ideals of graphs. For us, a graph means a finite simple undirected graph, i.e., it has no loops or multi-edges. Given a graph $G$, its corresponding incidence matrix $A_G$ has columns indexed by edges $e \in E(G)$, rows indexed by vertices $i \in V(G)$, and entries $A_{i,e} = 1$ if $i \in e$ and $A_{i,e} = 0$ if $i \notin e$. We note that every column of $A_G$ contains two ones. Thus, if we add up all the rows and divide by 2, we get the vector $\mathbf{1}_n$ meaning $A_G$ is homogeneous. Thus, throughout this section, we will use the restatement of the distance-reduction condition, namely Proposition \ref{prop:drredef}. We will write the toric ideal corresponding to a graph as $I_G:=I_{A_G}$ which now lives in the ring $R=k[e\,|\,e \in E(G)]$. We note that supports of vectors are now sets of edges of $G$.

In Section~\ref{sec: graphs markov bases}, we introduce some notation and state some key results about toric ideals of graphs.
Then, in Section~\ref{sec: graphs complete intersection} we prove the following conjecture in the case that $I_G$ is a complete intersection toric ideal.

First we recall that, for a integer matrix $A$, the circuits of $A$ have been shown to characterise the distance reduction property for complete intersection monomial curves \cite[Corollary~5.2]{CK_2024}. We conjecture that the same holds for toric ideals of graphs.

\begin{conjecture}\label{conj: distance reduction for graph ideals by circuits}
    Let $G$ be a simple graph and $M$ a minimal Markov basis for $I_G$. Then $M$ is distance-reducing if and only if $M$ reduces the distance of the circuits of $A_G$.
\end{conjecture}

We have verified this conjecture for all graphs on at most $7$ vertices and all graphs with $8$ vertices and at most $16$ edges, see the supplementary code \cite{githubM2DistanceReduction}.
In addition, we only check graphs which are connected, are not bipartite, and do not contain $K_4$ as an induced subgraph. This is justified by Propositions \ref{prop:conjccs}, \ref{prop:nobipartite} and \ref{prop:noK4s} respectively.


\begin{proposition} \label{prop:nobipartite}
Suppose $G$ is a bipartite graph. Then Conjecture~\ref{conj: distance reduction for graph ideals by circuits} holds for $G$. 
\end{proposition}

\begin{proof}
    By \cite[Theorem 3.13]{tatakis2025unimodulartoricidealsgraphs}, since $G$ is bipartite then $I_G$ is a unimodular toric ideal. Furthermore, by \cite[Proposition 8.18]{sturmfels1996grobner}, for a unimodular toric ideal, the set of circuits equals the Graver basis. By \cite[Theorem 7.4]{CK_2024}, if $M$ is a minimal Markov basis, $M$ is distance-reducing if and only if $M$ distance-reduces the Graver basis of $A_G$. Since the set of circuits is equal to the Graver basis then we immediately have that $M$ is distance-reducing if and only if $M$ distance-reduces the circuits of $A_G$.
\end{proof}

\begin{proposition} \label{prop:noK4s}
    Suppose $G$ is a graph that contains $K_4$ as an induced subgraph. Then Conjecture~\ref{conj: distance reduction for graph ideals by circuits} holds for $G$.
\end{proposition}

\begin{proof}
    We will show that, given a minimal Markov basis $M$ of $A_G$, we can always find a circuit of $\ker_\ZZ(A_G)$ which is not distance-reduced by $A_G$.

    Indeed, let us label the vertices of the $K_4$ as $1,2,3,4$. Then, the toric ideal $I_G$ contains the following 3 binomials, $e_{12}e_{34}-e_{13}e_{24}$, $e_{12}e_{34}-e_{14}e_{23}$ and $e_{13}e_{24}-e_{14}e_{23}$ which correspond to $z_1,z_2,z_3 \in \operatorname{Gr}(A_G)$ respectively. Any minimal generating set of $I_G$ contains at most two of these binomials since $\beta(z_2)=\beta(z_1)+\beta(z_3)$. So, up to renumbering our vertices assume $z_3 \not\in M$. Suppose $z_3$ is distance-reduced by some element in $M$,
    say $z \in M$. Note that, since $G$ has no multi-edges, $I_G$ does not contain any linear binomials, thus $\beta(z)$ has degree at least two. By part (\ref{item:drbouquet1}) of Definition \ref{def:drbouquet}, it follows that $\deg(\beta(z)) = 2$ and, moreover, the set of vertices that appear among indices of the variables in $\beta(z)$ is exactly $1234$. Thus, the only possibilities for $z$ are $z_1$ and $z_2$ with $z_1^- \le z_3^+$ and $z_2^- \le z_3^-$.
    However, by Proposition \ref{prop:drredef}, neither $z_1$ nor $z_2$ distance-reduce $z_3$ since $\operatorname{Supp}(z_1^+) \cap \operatorname{Supp}(z_3^-)$ and $\operatorname{Supp}(z_2^+) \cap \operatorname{Supp}(z_3^+)$ are empty. Thus, $z_3$ is not distance-reduced by $M$ and, furthermore, by \cite[Proposition 4.2]{Vil_1995}, $z_3$ is a circuit.

    Thus, if the graph $G$ has $K_4$ as a subgraph, for any minimal Markov basis $M$, the claim: $M$ is distance-reducing if and only if $M$ distance-reduces the circuits, always holds. This is because all minimal Markov bases do not distance-reduce at least one circuit.
\end{proof}

In order to show why we only need to look at connected graphs when considering Conjecture \ref{conj: distance reduction for graph ideals by circuits}, we need to show that when a graph $G$ has disconnected components, then the Graver basis, the circuits and the minimal Markov bases are partitioned between the disconnected components of $G$.
In the following, we will treat the ideals $I_H$ and $I_K$ and their binomials as though they live in the larger ring $k[e\,|\, e \in E(G)]$. Specifically, we take the generators for $I_H$ or $I_K$ and take those binomials to be the generators in the larger ring.

\begin{lemma}
\label{lem:disjointidealsgraver}
    For $G$ a graph, let $G=H \sqcup K$ for some disjoint graphs $H$ and $K$. Then, $I_G=I_H+I_K$ as well as $\operatorname{Gr}(I_G)=\operatorname{Gr}(I_H) \sqcup \operatorname{Gr}(I_K)$.
\end{lemma}

\begin{proof}
    Since $I_H \subseteq I_G$ and $I_K \subseteq I_G$, then $I_H+I_K \subseteq I_G$. Similarly, $\operatorname{Gr}(I_H) \subseteq \operatorname{Gr}(I_G)$, $\operatorname{Gr}(I_K) \subseteq \operatorname{Gr}(I_G)$ and $\operatorname{Gr}(I_H) \subseteq k[e \,|\, e \in E(H)]$ and $\operatorname{Gr}(I_K) \subseteq k[e \,|\, e \in E(K)]$. Thus, since $E(H) \cap E(K)$ is empty then so is $\operatorname{Gr}(I_H) \cap \operatorname{Gr}(I_K)$ and thus $ \operatorname{Gr}(I_H) \sqcup \operatorname{Gr}(I_K) \subseteq \operatorname{Gr}(I_G)$. 
    
    Now, take $x^{z^+}-x^{z^-} \in I_G$ for some $z \in \ker_\ZZ(A_G)$. Since $H$ and $K$ are disjoint, then $A_G$ is block diagonal with $A_G=A_H \oplus A_K$. In particular, we have a natural conformal decomposition of $z$ into $z=z_H+z_K$ where $z^+=z_H^+ + z_K^+$ and $z^-=z_H^- + z_K^-$. Then, 
    $$x^{z^+}-x^{z^-}=x^{z_H^+}x^{z_K^+}-x^{z_H^-}x^{z_K^-}=x^{z_K^+}(x^{z_H^+}-x^{z_H^-})+x^{z_H^-}(x^{z_K^+}-x^{z_K^-})$$
    lies in $I_H+I_K$. Since $I_G$ is generated by elements of the form $x^{z^+}-x^{z^-}$ for $z \in \ker_\ZZ(A_G)$ then $I_G \subseteq I_H+I_K$. Furthermore, if $\beta(z) \in \operatorname{Gr}(I_G)$ for some $z \in \operatorname{Gr}(A_G)$, then, since the Graver basis by definition contains elements with no proper conformal decomposition, either $z_H$ or $z_K$ is zero. Since $\beta(z)$ is primitive in a larger ring, it is also primitive in a smaller ring meaning $\beta(z) \in \operatorname{Gr}(I_H) \sqcup \operatorname{Gr}(I_K)$. Since $z$ was arbitrary then $\operatorname{Gr}(I_G) \subseteq \operatorname{Gr}(I_H) \sqcup \operatorname{Gr}(I_K)$.
\end{proof}

\begin{remark} \label{rem:disjointcomps}
    Now decompose $G$ into its connected components as $G=\bigsqcup_k H_k$ for some connected components $H_k$. By repeated use of Lemma \ref{lem:disjointidealsgraver}, we have the partition $\operatorname{Gr}(I_G)=\bigsqcup_k \operatorname{Gr}(I_{H_k})$. Then, for any $z \in \operatorname{Gr}(A_G)$, there is some $j$ such that $\beta(z) \in \operatorname{Gr}(I_{H_j})$ and some corresponding element $z' \in \ker_\ZZ(A_{H_j})$, which we now identify with $z$. This is justified since all ideal-theoretic properties (such as being primitive, a circuit, etc.) are carried over between $z$ and $z'$ since they have the same binomial $\beta(z)=\beta(z')$ in the ideal $I_G$. Thus, we can write $\operatorname{Gr}(A_G)=\bigsqcup_k \operatorname{Gr}(A_{H_k})$.
\end{remark}

\begin{corollary} \label{cor:disjointcircuits}
    For $G$ a graph where $G=\bigsqcup_k H_k$ for some connected components $H_k$ then $C(A_G)=\bigsqcup_k C(A_{H_k})$.
\end{corollary}

\begin{proof}
    Since the circuits of a toric ideal are always a subset of the Graver basis, (see \cite{sturmfels1996grobner}) then $C(A_G)=\bigsqcup_k C(A_{H_k})$ comes from $\operatorname{Gr}(I_G)=\bigsqcup_k \operatorname{Gr}(I_{H_k})$, following from Lemma \ref{lem:disjointidealsgraver}.
\end{proof}

\begin{corollary}
\label{cor:disjointmarkov}
    If $M$ is a minimal Markov basis of $I_G$, then there is a partition $M= \bigsqcup_k M_{H_k}$ where $M_{H_k}$ is a minimal Markov basis of $I_{H_k}$.
\end{corollary}

\begin{proof}
    Since any minimal Markov basis lies in the Graver basis (see \cite{sturmfels1996grobner}), the partition $\operatorname{Gr}(A_G)=\bigsqcup_k \operatorname{Gr}(A_{H_k})$ induces the partition $M= \bigsqcup_k M_{H_k}$ where $\{\beta(z) \,|\, z \in M_{H_k}\} \subseteq I_{H_k}$. Furthermore, by repeated use of Lemma \ref{lem:disjointidealsgraver}, $I_G=\sum_k I_{H_k}$ and, since $M$ is a Markov basis of $I_G$, then $\langle \beta(z) \,|\, z \in M_{H_k} \rangle = I_{H_k}$. Finally, $M$ is minimal so each $M_{H_k}$ is a minimal Markov basis of $I_{H_k}$.
\end{proof}

\begin{proposition} \label{prop:conjccs}
    For a graph $G$, if Conjecture~\ref{conj: distance reduction for graph ideals by circuits} holds for all of the connected components of $G$, then the conjecture also holds for $G$.
\end{proposition}




\begin{proof}
    
   Decompose $G$ into its connected components as $G=\bigsqcup_{k} H_k$ and suppose Conjecture~\ref{conj: distance reduction for graph ideals by circuits} holds for $H_k$ for all $k$. Seeking a contradiction, we suppose that the conjecture does not hold for $G$. Then, there is some minimal Markov basis $M$ for $I_G$ which distance-reduces the circuits $C(A_G)$ of $A_G$ but is not distance-reducing. By \cite[Theorem 7.4]{CK_2024}, since $M$ is not distance-reducing then $M$ does not distance-reduce the Graver basis $\operatorname{Gr}(A_G)$ so we can find $z \in \operatorname{Gr}(A_G) \backslash C(A_G)$ which is not distance-reduced by $M$. 
   
   As in Remark \ref{rem:disjointcomps}, there is a $j$ such that $\beta(z) \in \operatorname{Gr}(I_{H_j})$ and we treat $z$ as an element of $\ker_\ZZ(A_{H_j})$. In particular, $z \in \operatorname{Gr}(A_{H_j}) \backslash C(A_{H_j})$. Now, by Corollaries \ref{cor:disjointcircuits} and \ref{cor:disjointmarkov}, we have partitions $C(A_G)=\bigsqcup_k C(A_{H_k})$ and $M=\bigsqcup_k M_{H_k}$ where $M_{H_k}$ is a minimal Markov basis of $I_{H_j}$. In order to find a contradiction we will show that Conjecture \ref{conj: distance reduction for graph ideals by circuits} fails for $H_j$ by showing that $M_{H_j}$, which is a minimal Markov basis of $I_{H_j}$, distance-reduces the circuits $C(A_{H_j})$ but not $z$, which is not a circuit. 
   
   Take $c \in C(A_{H_j})$. Since $M$ distance-reduces $C(A)$, there is some $m \in M$ which distance-reduces $c$. Referring back to Remark \ref{rem:disjointcomps} again, we let $j'$ be such that $\beta(m) \in \operatorname{Gr}(I_{H_{j'}})$. If $j \neq j'$ then $\operatorname{Supp}(c) \cap \operatorname{Supp}(m)$ is empty, in particular, $\operatorname{Supp}(c^+) \cap \operatorname{Supp}(m^+)$ is empty. However, condition (\ref{item:drbouquet1}) of the distance-reduction condition, Definition \ref{def:drbouquet} tells us that, up to sign, $m^+ \le c^+$ and thus $\operatorname{Supp}(m^+) \subseteq \operatorname{Supp}(c^+)$. But $\operatorname{Supp}(c) \cap \operatorname{Supp}(m)$ is empty so we have a contradiction. Thus, $j=j'$ and $m \in \operatorname{Gr}(A_{H_j})$ so $m \in M_{H_j}$. In particular, $M_{H_j}$ distance-reduces $c$ and, since $c$ was arbitrary, then $M_{H_j}$ distance-reduces $C(A_{H_j})$. Finally, since $M$ does not distance-reduce $z$ then, neither does $M_{H_j}$ since it is a subset of $M$. Thus we have a minimal Markov basis $M_{H_j}$, which distance-reduces the circuits $C(A_{H_j})$ but is not distance-reducing since it fails to distance-reduce $z$. Thus Conjecture \ref{conj: distance reduction for graph ideals by circuits} fails for $H_j$.
\end{proof}



\subsection{Markov bases for toric ideals of graphs}\label{sec: graphs markov bases}



    

 
In this section, we will introduce the background on toric ideals of graphs. For a comprehensive overview of toric ideals of graphs see \cite[Section 2]{Reyes2012minimal}. In the following, we will let $G$ be a connected graph with vertex set $V(G)=\{v_1,\dots ,v_d\}$. However, by Proposition \ref{prop:conjccs}, the main theorem for this section, Theorem \ref{thm:CImain}, also holds for disconnected graphs.

A \textit{walk} is a finite sequence of the form
\begin{align*}
    w=(v_{i_0},e_{j_1},v_{i_1},e_{j_2},v_{i_2},\dots ,v_{i_{q-1}},e_{j_q},v_{i_{q}})
\end{align*}
where $q \in \NN$ and $e_{j_l}=\{v_{i_{l-1}},v_{i_l} \} \in E(G)$ for all $l \in [q]$. The \textit{length} of the walk $w$ is the number $q$ of edges in the walk. An \textit{even}/\textit{odd} walk is a walk of even/odd length. The walk $w$ is closed if $v_{i_0}=v_{i_q}$ and is a \textit{path} if the vertices $v_{i_0},\dots ,v_{i_q}$ are all distinct. A \textit{cycle} is a closed walk in which the vertices $v_{i_0},\dots ,v_{i_{q-1}}$ are distinct and we say that $w$ only visits the vertex $v_{i_0}=v_{i_q}$ once. We let $-w$ denote the reverse walk $(v_{i_{q}},e_{j_q},v_{i_{q-1}},\dots,v_{i_2},e_{j_2},v_{i_1},e_{j_1},v_{i_0} )$. Note that, although the graph $G$ has no multiple edges since it is simple, the same edge $e$ may appear more than once in a walk $w$. In this case, $e$ is called a \textit{multiple edge} of the walk $w$. We will denote the set of vertices of $w$ by $V(w)=\{v_{i_0},\dots v_{i_{q}}\}$ and the set of edges of $w$ by $E(w)=\{e_{j_1},\dots , e_{j_q}\}$. We then let $\mathbf{w}$ denote the \textit{induced graph of $w$} with vertices $V(w)$ and edges $E(w)$. 

Given an even closed walk
$$w=(v_{i_0},e_{j_1},v_{i_1},e_{j_2},v_{i_2},\dots ,v_{i_{2q-1}},e_{j_{2q}},v_{i_0})$$
of the graph $G$, we denote by $E^+(w)$ and $E^-(w)$ the monomials
$$E^+(w)=\prod_{l=1}^q e_{j_{2l-1}}, \quad E^-(w)=\prod_{l=1}^q e_{j_{2l}}.$$
A walk $w'$ is a \textit{subwalk} of $w$ if $E^+(w')E^-(w')| E^+(w)E^-(w)$. Then, we denote by $B_w$ the binomial
$$B_w=E^+(w)-E^-(w)$$
belonging to the toric ideal $I_G$. By  \cite[Proposition 3.1]{Vil_1995}, $I_G$ is generated by binomials of this form and, by \cite[Lemma 2.1]{ohsugi-hibi-quad-binoms}, all elements of $\operatorname{Gr}(I_G)$ are of this form. Note that different walks may correspond to the same binomial, however all these walks give the same induced graph $\mathbf{w}$ so we will identify them. In addition, given a walk $w$, since $B_w \in I_{G}$, then there is $z \in \ker_\ZZ(A_G)$ such that $\beta(z)=B_w$ and we say $z$ corresponds to the walk $w$. We will say that the walk $w$ is primitive if the corresponding element $z \in \ker_\ZZ(A_G)$ lies in the Graver basis $\operatorname{Gr}(A_G)$, which is equivalent to $\beta(z)$ being primitive, see \cite[Chapter 4]{sturmfels1996grobner} and \cite[Lemmas~1.1 and 3.1]{ohsugi-hibi-quad-binoms}. On the other hand, given $z \in \operatorname{Gr}(A_G)$, we can write $\beta(z)=B_w$ for some even closed walk $w$ and $z$ corresponds to the walk $w$. 

In fact, a graph-theoretic classification for the walks corresponding to elements of the Graver basis is given by Reyes, Tatakis and Thoma, which we will restate in Theorem \ref{thm:reyesprimitive}. Furthermore, since $C(A_G) \subseteq \operatorname{Gr}(A_G)$, all circuits correspond to even closed walks on the graph $G$ and a complete classification of circuits in terms of walks on graphs was given by Villarreal in \cite[Proposition 4.2]{Vil_1995}. For completeness we restate this now as Theorem \ref{thm:villarealcircuits}.

\begin{theorem} \cite[Proposition 4.2]{Vil_1995} \label{thm:villarealcircuits}
    For $G$ a simple connected graph, $z \in \operatorname{Gr}(A_G)$ is a circuit if and only if the corresponding even closed walk $w$ is
    \begin{enumerate}
        \item\label{item:circ1} an even cycle,
        \item\label{item:circ2} two odd cycles intersecting at exactly one vertex, or
        \item\label{item:circ3} two vertex disjoint odd cycles joined by a path.
    \end{enumerate}
\end{theorem}

To state Reyes, Tatakis and Thoma's classification of Graver basis elements, we first introduce some more graph-theoretic concepts, see \cite[Section 3]{Reyes2012minimal}.

A \textit{cut edge/vertex} is an edge/vertex of the graph $G$ whose removal increases the number of connected components of the remaining subgraph. A graph is \textit{biconnected} if it is connected and does not contain a cut vertex. A \textit{block} of $G$ is a maximal biconnected subgraph of $G$. For a primitive even closed walk $w$, we note that $E(w)$ is partitioned into those edges which divide $E^+(w)$ and those which divide $E^-(w)$. If an edge $e$ of $w$ divided both $E^+(w)$ and $E^-(w)$ then $B_w$ would not be primitive since $\frac{B_w}{e} \in I_G$. Then, a \textit{sink} of a block $B$ of $\mathbf{w}$ is a common vertex of edges $e,e' \in E(w)$ which both belong to the block $B$ and where $ee'$ divides either $E^+(w)$ or $E^-(w)$. Note that $e$ and $e'$ don't have to be distinct. For example, if $e=\{v_1,v_2\}$ is a cut edge of $w$, then, since $w$ is closed, $e$ appears twice in the walk $w$. Also, since $w$ is primitive then $e^2$ divides either $E^+(w)$ or $E^-(w)$ so $v_1,v_2$ are both sinks.

\begin{theorem} \cite[Theorem 3.1]{Reyes2012minimal} \label{thm:reyesprimitive}
    Let $G$ be a graph and $w$ an even closed walk of $G$. Then, the walk $w$ is primitive if and only if
    \begin{enumerate}
        \item\label{item:prim1} every block of the induced graph $\mathbf{w}$ is a cycle or a cut edge,
        \item\label{item:prim2} every multiple edge of the walk $w$ is a double edge of the walk and a cut edge of $\mathbf{w}$,
        \item\label{item:prim3} every cut vertex of $\mathbf{w}$ belongs to exactly two blocks and is a sink of both.
    \end{enumerate}
\end{theorem}

The primitivity of an even closed walk $w$ of $G$ restricts the amount of times each vertex of $G$ can be visited by $w$. We will need this result for Theorem \ref{thm:CImain} so we provide a proof.

\begin{lemma} \label{lem:vertex<3}
    Let $G$ be a graph and $w$ a primitive even closed walk of $G$. Then, each vertex of $G$ is visited at most twice by $w$ and exactly twice if and only if it is a cut vertex of $\mathbf{w}$.
\end{lemma}

\begin{proof}
    Since $w$ is primitive, we can use conditions (\ref{item:prim1}), (\ref{item:prim2}) and (\ref{item:prim3}) of Theorem \ref{thm:reyesprimitive}. 

    We first note that if a vertex $v$ in the walk $w$ is not a cut vertex of $\mathbf{w}$, then it is only visited once by $w$. To see this, we notice that $v$ only appears in one block $B$ of $w$. Since both vertices in a block which is a cut edge are cut vertices then, by condition (\ref{item:prim1}), $B$ is a cycle. In addition, $w$ only traverses edges in $B$ once otherwise, by condition (\ref{item:prim2}), the edge would be a cut edge and be its own block. Thus, since $v$ only appears in the block $B$, $v$ is only visited once by $w$.
    
    We now prove the first part of the statement. Suppose there is a vertex of $G$, $v$, which is visited at least $3$ times by $w$. By our previous observation then $v$ is a cut vertex of $\mathbf{w}$. Thus, by condition (\ref{item:prim3}), $v$ belongs to exactly $2$ blocks of $w$, say $B$ and $B'$, and is a sink of both. Since $v$ is a sink, any time $w$ visits $v$, it switches the block it is travelling through. Consider the set of edges of $B$ adjacent to $v$. The number of times each is visited by $w$ must add up to at least $3$. If one of them is a double edge of $w$ then, by condition (\ref{item:prim2}), the block $B$ is that cut edge in which case this sum would be $2$. By condition (\ref{item:prim2}), each edge of $w$ is traversed at most twice, so each edge of $B$ adjacent to $v$ is only traversed at most once by $w$. Thus, there must be at least $3$ edges of $B$ which are adjacent to $v$. This contradicts condition (\ref{item:prim1}) since then $B$ is not a cycle or a cut edge. Thus, any vertex of $G$ is visited at most twice by $w$.

   We now prove the second part of the statement. The ($\Rightarrow$) direction follows by our observation that, if $v$ is not a cut vertex of $\mathbf{w}$ then $v$ is only visited once by $w$. For the ($\Leftarrow$) direction, if $v$ is a cut vertex, then, as noted before, every time $w$ visits $v$, it switches the block it is travelling through. Since $w$ is a closed walk, then $w$ must visit $v$ at least twice to get back to where the walk started.
\end{proof}

As in \cite[Section 4]{tatakis2011completeintersectiontoricideals}, we define the \textit{block tree} of $G$, $B(G)$, which is the bipartition $(\mathbb{B},\mathbb{S})$ where $\mathbb{B}$ is the set of blocks of $G$ and $\mathbb{S}$ is the set of cut vertices of $G$. Then, $\{B,v\}$ is an edge of $B(G)$ if and only if $v \in B$. The leaves of the block tree and are always blocks and are called \textit{end blocks}. We note that, since end blocks are leaves, they contain unique cut vertices.

We recall that, given a matrix $A \in \ZZ^{d \times n}$, the toric ideal $I_A$ is a \textit{complete intersection toric ideal} if the number of minimal generators of $I_A$ is equal to the height of the toric ideal ${\rm ht}(I_A) = n - {\rm rk}(A)$. 

Tatakis and Thoma give a graph-theoretic classification of complete intersection toric ideals of graphs in \cite{tatakis2011completeintersectiontoricideals}. We now restate one of their key results about the minimal generators of such ideals which is the crux of the proof of Theorem \ref{thm:CImain}.

\begin{theorem} \cite[Theorem 5.4]{tatakis2011completeintersectiontoricideals} \label{thm:tatakisCI}
    Let $G$ be a simple connected graph where $I_G$ is a complete intersection. Then, all minimal generators of $I_G$ except at most one correspond to a walk which is an even cycle.
\end{theorem}

We present  below an example of a complete intersection toric ideal of a graph and discuss its distance reduction properties, which will help with gaining intuition and clarifying the tools used in the proof of Theorem~\ref{thm:CImain}.

\begin{example}\label{ex: CI graph}
    We present an example of a graph which has a complete intersection toric ideal. We denote the graph $G_1$ and visualise it in Figure \ref{fig:mainG1}.
    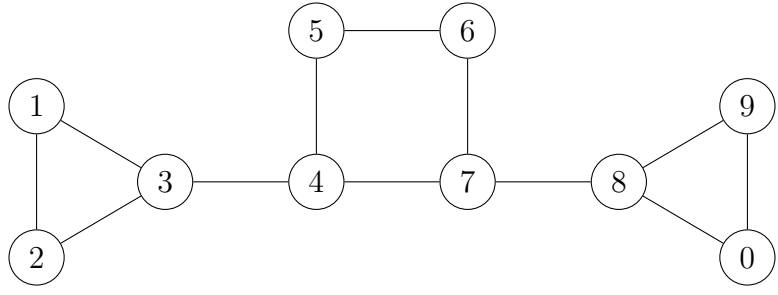
\begin{figure}[t]
    \begin{tikzpicture} [scale=1]
	\begin{pgfonlayer}{nodelayer}
        \node[shape=circle,draw=black] (1) at (-1.7,2) {$1$};
        \node[shape=circle,draw=black] (2) at (-1.7,0) {$2$};
        \node[shape=circle,draw=black] (3) at (0,1) {$3$};
        \node[shape=circle,draw=black] (4) at (2,1) {$4$};
        \node[shape=circle,draw=black] (5) at (2,3) {$5$};
        \node[shape=circle,draw=black] (6) at (4,3) {$6$};
        \node[shape=circle,draw=black] (7) at (4,1) {$7$};
        \node[shape=circle,draw=black] (8) at (6,1) {$8$};
        \node[shape=circle,draw=black] (9) at (7.7,2) {$9$};
        \node[shape=circle,draw=black] (0) at (7.7,0) {$0$};
	\end{pgfonlayer}
	\begin{pgfonlayer}{edgelayer}
		\draw (1) to (2);
        \draw (1) to (3);
        \draw (2) to (3);
        \draw (3) to (4);
        \draw (4) to (5);
        \draw (5) to (6);
        \draw (6) to (7);
        \draw (4) to (7);
        \draw (7) to (8);
        \draw (8) to (9);
        \draw (9) to (0);
        \draw (8) to (0);
	\end{pgfonlayer}
    \end{tikzpicture}
    \centering
    \caption{The graph $G_1$ in Example~\ref{ex: CI graph}.} \label{fig:mainG1}
\end{figure}

The corresponding incidence matrix of $G_1$ is given by
$$A_{G_1}=\begin{bmatrix}
1 & 1 & 0 & 0 & 0 & 0 & 0 & 0 & 0 & 0 & 0 & 0 \\
1 & 0 & 1 & 0 & 0 & 0 & 0 & 0 & 0 & 0 & 0 & 0 \\
0 & 1 & 1 & 1 & 0 & 0 & 0 & 0 & 0 & 0 & 0 & 0 \\
0 & 0 & 0 & 1 & 1 & 1 & 0 & 0 & 0 & 0 & 0 & 0 \\
0 & 0 & 0 & 0 & 1 & 0 & 1 & 0 & 0 & 0 & 0 & 0 \\
0 & 0 & 0 & 0 & 0 & 0 & 1 & 1 & 0 & 0 & 0 & 0 \\
0 & 0 & 0 & 0 & 0 & 1 & 0 & 1 & 1 & 0 & 0 & 0 \\
0 & 0 & 0 & 0 & 0 & 0 & 0 & 0 & 1 & 1 & 1 & 0 \\
0 & 0 & 0 & 0 & 0 & 0 & 0 & 0 & 0 & 1 & 0 & 1 \\
0 & 0 & 0 & 0 & 0 & 0 & 0 & 0 & 0 & 0 & 1 & 1 
\end{bmatrix} \in \ZZ^{10 \times 12}.$$

Using the \textit{AllMarkovBases} package \cite{CM_2025} for Macaulay2 \cite{M2}, we compute that $A_{G_1}$ has a unique minimal Markov basis, given by
$$M=\begin{bmatrix}z_1\\z_2\end{bmatrix}=\begin{bmatrix}
    0&0&0&0&1&-1&-1&1&0&0&0&0\\
    1&-1&-1&2&0&-2&0&0&2&-1&-1&1
\end{bmatrix}.$$
Since $n=12$ and ${\rm rk(A_{G_1})}=10$  then ${\rm ht}(I_{A_{G_1}})=n-{\rm rk}(A_{G_1})=2$. Thus, since $M$ has two elements then $I_{G_1}$ is a complete intersection toric ideal. We note that we could also have used \cite[Theorem 5.5]{tatakis2011completeintersectiontoricideals} which gives necessary and sufficient conditions for $I_{G_1}$ to be complete intersection.

We denote an edge $e=\{i,j\}$ of $G_1$ as $e_{ij}$ and so we miss out the edges when writing down walks around $G_1$. Then $\beta(z_1)=e_{45}e_{67}-e_{47}e_{56}=B_{w_1}$ where we define the walk $w_1:=(4,5,6,7,4)$. Similarly, $\beta(z_2)=e_{12}e_{34}^2e_{78}^2e_{90}-e_{13}e_{23}e_{47}^2e_{89}e_{80}=B_{w_2}$ where we define the walk $w_2:=(1,2,3,4,7,8,9,0,8,7,4,3,1)$.

The graphs $\mathbf{w_1}$ and $\mathbf{w_2}$ are visualised in Figure \ref{fig:w1w2}. We notice that $w_1$ is an even cycle and $w_2$ is two vertex disjoint odd cycles, $(1,2,3,1)$ and $(8,9,0,8)$, joined by the path $(3,4,7,8)$. Thus, by Theorem \ref{thm:villarealcircuits}, $z_1$ and $z_2$ are circuits of $A_{G_1}$ and, since every minimal generator of $I_{G_1}$ is a circuit, then $I_{G_1}$ is a circuit ideal. This agrees with \cite[Theorem 5.1]{tatakis2011completeintersectiontoricideals} according to which every complete intersection toric ideal coming from a graph is a circuit ideal. We also note that only one element of $M$ does not correspond to an even cycle so our analysis aligns with Theorem \ref{thm:tatakisCI} as well.

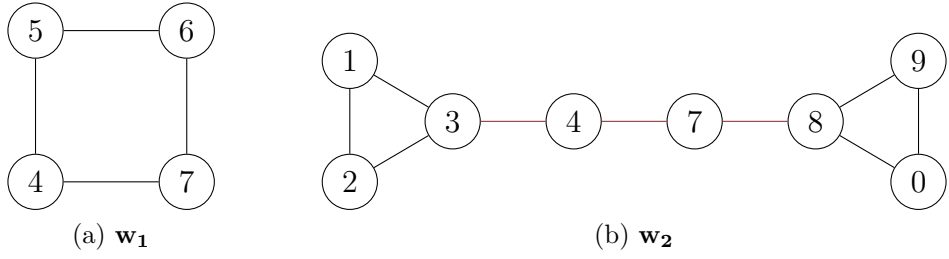
\begin{figure}
    \subfloat[\centering $\mathbf{w_1}$]{{
    \begin{tikzpicture} [scale=1]
	\begin{pgfonlayer}{nodelayer}
        \node[shape=circle,draw=black] (4) at (2,1) {$4$};
        \node[shape=circle,draw=black] (5) at (2,3) {$5$};
        \node[shape=circle,draw=black] (6) at (4,3) {$6$};
        \node[shape=circle,draw=black] (7) at (4,1) {$7$};
	\end{pgfonlayer}
	\begin{pgfonlayer}{edgelayer}
        \draw (4) to (5);
        \draw (5) to (6);
        \draw (6) to (7);
        \draw (4) to (7);
	\end{pgfonlayer}
    \end{tikzpicture}
    }}%
    \hspace{1cm}
    \subfloat[\centering $\mathbf{w_2}$]{{
    \begin{tikzpicture} [scale=0.8]
	\begin{pgfonlayer}{nodelayer}
        \node[shape=circle,draw=black] (1) at (-1.7,2) {$1$};
        \node[shape=circle,draw=black] (2) at (-1.7,0) {$2$};
        \node[shape=circle,draw=black] (3) at (0,1) {$3$};
        \node[shape=circle,draw=black] (4) at (2,1) {$4$};
        \node[shape=circle,draw=black] (7) at (4,1) {$7$};
        \node[shape=circle,draw=black] (8) at (6,1) {$8$};
        \node[shape=circle,draw=black] (9) at (7.7,2) {$9$};
        \node[shape=circle,draw=black] (0) at (7.7,0) {$0$};
	\end{pgfonlayer}
	\begin{pgfonlayer}{edgelayer}
		\draw (1) to (2);
        \draw (1) to (3);
        \draw (2) to (3);
        \draw[red] (3) to (4);
        \draw[red] (4) to (7);
        \draw[red] (7) to (8);
        \draw (8) to (9);
        \draw (9) to (0);
        \draw (8) to (0);
	\end{pgfonlayer}
    \end{tikzpicture}
    }}%
    \centering
    \caption{The graphs $\mathbf{w_1},\mathbf{w_2}$ corresponding to walks $w_1,w_2$ respectively. Double edges of the walks are given by red edges.} \label{fig:w1w2}
\end{figure}

Finally, we look at the Graver basis of $A_{G_1}$, which we represent as a matrix with rows given by its elements:

$$\operatorname{Gr}(A_{G_1})=\begin{bmatrix}z_1\\z_2\\z_3\\z_4\end{bmatrix}=\begin{bmatrix}
    0&0&0&0&1&-1&-1&1&0&0&0&0\\
    1&-1&-1&2&0&-2&0&0&2&-1&-1&1\\
    1&-1&-1&2&-2&0&2&-2&2&-1&-1&1\\
    1&-1&-1&2&-1&-1&1&-1&2&-1&-1&1
\end{bmatrix}.$$

\begin{figure}
    \subfloat[\centering $\mathbf{w_3}$]{{
    \begin{tikzpicture} [scale=0.7]
	\begin{pgfonlayer}{nodelayer}
        \node[shape=circle,draw=black] (1) at (-1.7,2) {$1$};
        \node[shape=circle,draw=black] (2) at (-1.7,0) {$2$};
        \node[shape=circle,draw=black] (3) at (0,1) {$3$};
        \node[shape=circle,draw=black] (4) at (2,1) {$4$};
        \node[shape=circle,draw=black] (5) at (2,3) {$5$};
        \node[shape=circle,draw=black] (6) at (4,3) {$6$};
        \node[shape=circle,draw=black] (7) at (4,1) {$7$};
        \node[shape=circle,draw=black] (8) at (6,1) {$8$};
        \node[shape=circle,draw=black] (9) at (7.7,2) {$9$};
        \node[shape=circle,draw=black] (0) at (7.7,0) {$0$};
	\end{pgfonlayer}
	\begin{pgfonlayer}{edgelayer}
		\draw (1) to (2);
        \draw (1) to (3);
        \draw (2) to (3);
        \draw[red] (3) to (4);
        \draw[red] (4) to (5);
        \draw[red] (5) to (6);
        \draw[red] (6) to (7);
        \draw[red] (7) to (8);
        \draw (8) to (9);
        \draw (9) to (0);
        \draw (8) to (0);
	\end{pgfonlayer}
    \end{tikzpicture}
    }}%
    \hspace{0.5cm}
    \subfloat[\centering $\mathbf{w_4}$]{{
    \begin{tikzpicture} [scale=0.7]
	\begin{pgfonlayer}{nodelayer}
        \node[shape=circle,draw=black] (1) at (-1.7,2) {$1$};
        \node[shape=circle,draw=black] (2) at (-1.7,0) {$2$};
        \node[shape=circle,draw=black] (3) at (0,1) {$3$};
        \node[shape=circle,draw=black] (4) at (2,1) {$4$};
        \node[shape=circle,draw=black] (5) at (2,3) {$5$};
        \node[shape=circle,draw=black] (6) at (4,3) {$6$};
        \node[shape=circle,draw=black] (7) at (4,1) {$7$};
        \node[shape=circle,draw=black] (8) at (6,1) {$8$};
        \node[shape=circle,draw=black] (9) at (7.7,2) {$9$};
        \node[shape=circle,draw=black] (0) at (7.7,0) {$0$};
	\end{pgfonlayer}
	\begin{pgfonlayer}{edgelayer}
		\draw (1) to (2);
        \draw (1) to (3);
        \draw (2) to (3);
        \draw[red] (3) to (4);
        \draw (4) to (5);
        \draw (5) to (6);
        \draw (6) to (7);
        \draw (4) to (7);
        \draw[red] (7) to (8);
        \draw (8) to (9);
        \draw (9) to (0);
        \draw (8) to (0);
	\end{pgfonlayer}
    \end{tikzpicture}
    }}%
    \centering
    \caption{The graphs $\mathbf{w_3},\mathbf{w_4}$ corresponding to walks $w_3,w_4$ respectively. Double edges of the walks are given by red edges.} \label{fig:w3w4}
\end{figure}
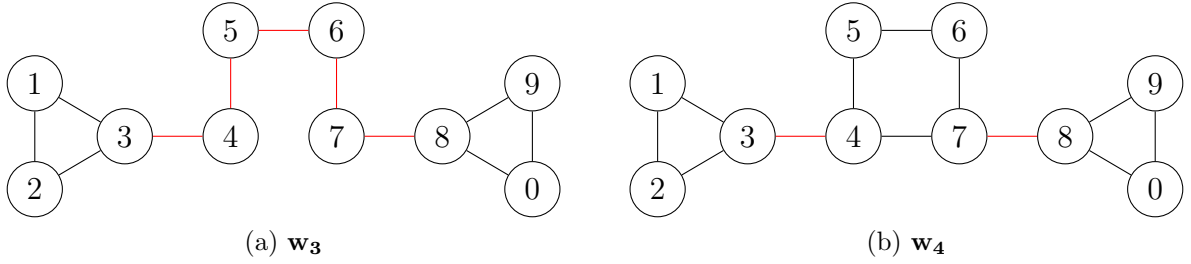

We then have $\beta(z_3)=e_{12}e_{34}^2e_{56}^2e_{78}^2e_{90}-e_{13}e_{23}e_{45}^2e_{67}^2e_{89}e_{80}=B_{w_3}$ 
where we define $w_3:=(1,2,3,4,5,6,7,8,9,0,8,7,6,5,4,3,1)$ and $\beta(z_4)=e_{12}e_{34}^2e_{56}e_{78}^2e_{90}-e_{13}e_{23}e_{45}e_{47}e_{67}e_{89}e_{80}=B_{w_4}$ 
where $w_4:=(1,2,3,4,5,6,7,8,9,0,8,7,4,3,1)$. The graphs $\mathbf{w_3}$ and $\mathbf{w_4}$ are visualised in Figure \ref{fig:w3w4}. As with $w_2$, the walk $w_3$ is made up of two vertex disjoint odd cycles joined by a path so $z_3$ is a circuit by Theorem \ref{thm:villarealcircuits}. However, $w_4$ does not fulfil any of the parts of Theorem \ref{thm:villarealcircuits} and thus $z_4$ is not a circuit.

Finally, we can confirm that $w_1,w_2,w_3,w_4$ all fulfil all three conditions of Theorem \ref{thm:reyesprimitive}. We do this explicitly for $w_4$. The blocks of $w_4$ are given by the induced graphs on the vertices $\{1,2,3\}$, $\{3,4\}$, $\{4,5,6,7\}$, $\{7,8\}$ and $\{8,9,0\}$ and we label these $B_1,B_2,B_3,B_4,B_5$ respectively. Condition (\ref{item:prim1}) is fulfilled since the blocks $B_1,B_3,B_5$ are cycles and $B_2,B_4$ are cut edges. Condition (\ref{item:prim2}) is fulfilled since the only multiple edges are $e_{34}$ and $e_{78}$ which are double edges and these are also cut edges of $\mathbf{w_4}$. For condition (\ref{item:prim3}), we mark the edges of $\mathbf{w_4}$ with pluses and minuses depending on whether that edge appears in $E^+(w_4)$ or $E^-(w_4)$ as in Figure \ref{fig:w4pm}. From this we can see that the cut vertices of $\mathbf{w_4}$ which are $3,4,7$ and $8$ belong to exactly two blocks each. Furthermore, edges in the same block, which are adjacent to these vertices, have the same sign in Figure \ref{fig:w4pm} and thus $3,4,7$ and $8$ are all sinks.

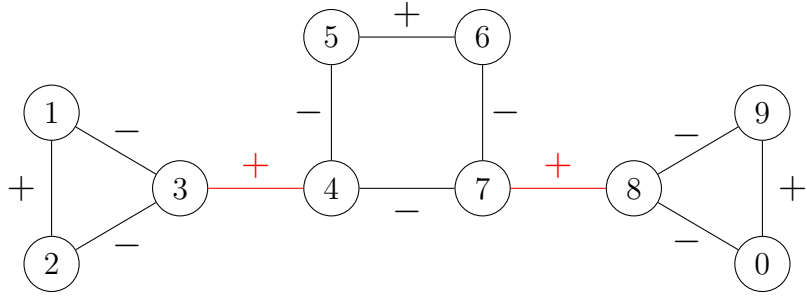
\begin{figure}[h] 
    \begin{tikzpicture} [scale=1]
	\begin{pgfonlayer}{nodelayer}
        \node[shape=circle,draw=black] (1) at (-1.7,2) {$1$};
        \node[shape=circle,draw=black] (2) at (-1.7,0) {$2$};
        \node[shape=circle,draw=black] (3) at (0,1) {$3$};
        \node[shape=circle,draw=black] (4) at (2,1) {$4$};
        \node[shape=circle,draw=black] (5) at (2,3) {$5$};
        \node[shape=circle,draw=black] (6) at (4,3) {$6$};
        \node[shape=circle,draw=black] (7) at (4,1) {$7$};
        \node[shape=circle,draw=black] (8) at (6,1) {$8$};
        \node[shape=circle,draw=black] (9) at (7.7,2) {$9$};
        \node[shape=circle,draw=black] (0) at (7.7,0) {$0$};
        \node[] at (-2.1,1) {\large $+$};
        \node[] at (-0.7,1.75) {\large $-$};
        \node[] at (-0.7,0.25) {\large $-$};
        \node[] at (1,1.3) {\large \textcolor{red}{$+$}};
        \node[] at (1.7,2) {\large $-$};
        \node[] at (3,3.3) {\large $+$};
        \node[] at (3,0.7) {\large $-$};
        \node[] at (4.3,2) {\large $-$};
        \node[] at (5,1.3) {\large \textcolor{red}{$+$}};
        \node[] at (6.7,1.7) {\large $-$};
        \node[] at (6.7,0.3) {\large $-$};
        \node[] at (8.1,1) {\large $+$};
	\end{pgfonlayer}
	\begin{pgfonlayer}{edgelayer}
		\draw (1) to (2);
        \draw (1) to (3);
        \draw (2) to (3);
        \draw[red] (3) to (4);
        \draw (4) to (5);
        \draw (5) to (6);
        \draw (6) to (7);
        \draw (4) to (7);
        \draw[red] (7) to (8);
        \draw (8) to (9);
        \draw (9) to (0);
        \draw (8) to (0);
	\end{pgfonlayer}
    \end{tikzpicture}
    \centering
    \caption{The graph $\mathbf{w_4}$ marked with pluses and minuses corresponding to whether that edge divides $E^+(w_4)$ or $E^-(w_4)$. Red edges still correspond to double edges which both have the same sign so we only mark these in red as well.} \label{fig:w4pm}
\end{figure}

Finally, we ask about the distance-reducing properties of $M$. Since $z_1,z_2 \in M$ then $M$ trivially distance-reduces $z_1$ and $z_2$. In addition, $z_3$ and $z_4$ are both distance-reduced by $z_1$ and $z_2$ and we can use Proposition \ref{prop:drredef} to confirm this. In the case of $z_3$, we have $z_1^+ \le z_3^-$ and $\operatorname{Supp}(z_1^-) \cap \operatorname{Supp}(z_3^+)=\{e_{56}\}$ as well as $z_2^+ \le z_3^+$ and $\operatorname{Supp}(z_2^-)\cap \operatorname{Supp}(z_3^-)=\{e_{13},e_{23},e_{89},e_{80}\}$. In the case of $z_4$, we have $z_1^+ \le z_4^-$ and $\operatorname{Supp}(z_1^-) \cap \operatorname{Supp}(z_4^+)=\{e_{67}\}$ as well as $z_2^+ \le z_4^+$ and $\operatorname{Supp}(z_2^-) \cap \operatorname{Supp}(z_4^-)=\{e_{13},e_{23},e_{89},e_{80}\}$. Thus, $M$ distance-reduces the circuits of $A_{G_1}$ as well as the Graver basis $\operatorname{Gr}(A_{G_1})$. By \cite[Theorem 7.4]{CK_2024}, since $M$ distance-reduces $\operatorname{Gr}(A_{G_1})$ then $M$ is distance-reducing. In particular, for this example the statement, $M$ is distance-reducing if and only if $M$ distance-reduces the circuits, holds.
\end{example}

\subsection{Complete intersection toric ideals of graphs}\label{sec: graphs complete intersection}

The following Theorem is Conjecture \ref{conj: distance reduction for graph ideals by circuits} in the case that $I_G$ is a complete intersection.

\begin{theorem} \label{thm:CImain}
    Let $G$ be a simple connected graph where $I_{G}$ is a complete intersection and let $M$ be a minimal Markov basis for $I_{G}$. Then, $M$ is distance-reducing if and only if $M$ distance-reduces the circuits of $A_G$.
\end{theorem}

To prove this result, we require two technical lemmas.

\begin{lemma} \label{lem:cycleDR}
    Let $G$ be a simple, connected graph and let $\hat z, y,z \in \operatorname{Gr}(A_G)$ where $\operatorname{Supp}(y^+) \subseteq \operatorname{Supp}(z^+)$ and $\operatorname{Supp}(y^-) \subseteq \operatorname{Supp}(z^-)$. Then, if $y$ is distance-reduced by $\hat z$, where the walk corresponding to $\hat z$ on the graph $G$ is a cycle, then $z$ is also distance-reduced by $\hat z$.
\end{lemma}

\begin{proof}
    Since $y$ is distance-reduced by $\hat z$, then, by Proposition \ref{prop:drredef}, up to switching signs, $\hat z^+ \le y^+$ and $\operatorname{Supp}(\hat z^-) \cap \operatorname{Supp}(y^-)$ is non-empty. 
    
    Since $\hat z^+ \le y^+$, then $\operatorname{Supp}(\hat z^+) \subseteq \operatorname{Supp}(y^+)$ and we know $\operatorname{Supp}(y^+) \subseteq \operatorname{Supp}(z^+)$ so $\operatorname{Supp}(\hat z^+) \subseteq \operatorname{Supp}(z^+)$. But, $\hat z$ corresponds to a cycle on the graph $G$, so $(\hat z)_e \in \{-1,0,1\}$ for all $e \in E(G)$. Now, for $e \in \operatorname{Supp}(\hat z^+)$, $(\hat z)_e=1$. In addition, since $\operatorname{Supp}(\hat z^+) \subseteq \operatorname{Supp}(z^+)$, then $e \in \operatorname{Supp}(z^+)$, so $(z^+)_e \ge 1$. This means that $(\hat z^+)_e=1 \le (z^+)_e$. Finally, if $e \not\in \operatorname{Supp}(\hat z^+)$, then $(\hat z^+)_e=0 \le (z^+)_e$ since $z^+ \in \NN^n$. Thus, $\hat z^+ \le z^+$. 
    
    In addition, since $\operatorname{Supp}(\hat z^-) \cap \operatorname{Supp}(y^-)$ is non-empty, we can find some $e \in \operatorname{Supp}(\hat z^-) \cap \operatorname{Supp}(y^-)$. Since $\operatorname{Supp}(y^-) \subseteq \operatorname{Supp}(z^-)$, then $e \in \operatorname{Supp}(\hat z^-) \cap \operatorname{Supp}(z^-)$ meaning $ \operatorname{Supp}(\hat z^-) \cap \operatorname{Supp}(z^-)$ is non-empty. Thus, both conditions of Proposition \ref{prop:drredef} have been fulfilled, so $z$ is distance-reduced by $\hat z$.
\end{proof}

\begin{lemma} \label{lem:vertexin}
    Let $G$ be a simple, connected graph and let $\hat z,z \in \operatorname{Gr}(A_G)$ where $z$ is distance-reduced by $\hat z$. If $\hat w,w$ are walks corresponding to $\hat z,z$ respectively, then $V(\hat w) \subseteq V(w)$.
\end{lemma}

\begin{proof}
    Since $z$ is distance-reduced by $\hat z$, by Proposition \ref{prop:drredef}, up to switching signs, we have $\hat z^+ \le z^+$. Now, take a vertex $v \in V(\hat w)$. $v$ is adjacent to at least one edge $e \in E(\hat w)$ which divides $E^+(\hat w)$. Since $\hat z^+ \le z^+$, then $E^+(\hat w) | E^+(w)$ so $e|E^+(w)$. Thus, $e \in E(w)$ and $v \in V(w)$. Since $v$ was an arbitrary vertex, then $V(\hat w) \subseteq V(w)$.
\end{proof}

We are now ready to give a proof of Theorem~\ref{thm:CImain}. The strategy of this proof is to assume that a minimal Markov basis $M$ distance-reduces all the circuits of $\ker_\ZZ(A_G)$ and then, given any $z$ in the Graver basis, find an element of $M$ which distance-reduces $z$. In practice, we assume $M$ does not distance-reduce $z$ and come to a contradiction and this is enough to imply Theorem \ref{thm:CImain}. We start by decomposing the walk $w$ corresponding to $z$ into a number of subwalks $w_i$ which all correspond to circuits $z_i$ and thus are all distance-reduced by $M$. Then, by Theorem \ref{thm:tatakisCI}, all but at most one of the elements of $M$ correspond to cycles and when one of these cycles distance-reduces a circuit $z_i$, by Lemma \ref{lem:cycleDR}, it also distance-reduces $z$. Thus, in this case, $M$ distance-reduces $z$ and we are done. The rest of the proof, which is the majority, is then dedicated to the case that $M$ contains the element which is not a cycle and this element distance-reduces all the circuits $z_i$. This allows us to narrow down the structure of $w$ until we eventually find another element of $M$ which distance-reduces $z$ or we come to a contradiction.

\begin{proof}[Proof of Theorem~\ref{thm:CImain}.]
    The ($\Rightarrow$) direction follows immediately since, if $M$ is distance-reducing then it distance-reduces everything, including the circuits. For the ($\Leftarrow$) direction, by \cite[Proposition 7.2]{CK_2024}, to check if $M$ is distance-reducing, it suffices to check if $M$ distance-reduces all the elements of the Graver basis, $\operatorname{Gr}(A_G)$. So, suppose $M$ distance-reduces all the circuits of $A_G$ but fails to distance-reduce some $z \in \operatorname{Gr}(A_G)$, where we aim to prove a contradiction. Since $z$ is primitive, we can let $w$ denote the primitive even closed walk in the graph $G$ corresponding to $z$. Now, if $w$ is made up of $1$ block, then, since $w$ is primitive, by Theorem \ref{thm:reyesprimitive}, $w$ is an even cycle. However, by part (\ref{item:circ1}) of Theorem \ref{thm:villarealcircuits}, then $z$ is a circuit and, $z$ is distance-reduced by $M$ by assumption, which is a contradiction since $M$ does not distance-reduce $z$.

       Therefore, we can assume $w$ is made up of at least 2 blocks. Let $B_1,\dots,B_k$ denote the end blocks of the block tree $B(\mathbf{w})$, labelled whilst travelling around the walk $w$. Since $w$ is made up of at least $2$ blocks and any tree on at least $2$ vertices has at least $2$ leaves, then $k \ge 2$. Each end block $B_i$ is a cycle and contains a unique cut vertex, $v_i$, and we let $c_i$ denote the closed walk around $B_i$ starting and ending at $v_i$. Since the $v_i$ are sinks, then the $c_i$ are odd cycles for all $i \in [k]$. Thus, we can write 
    $$w=(v_1,c_1,v_1,p_1,v_2,c_2,v_2,p_2,\dots,v_k,c_k,v_k,p_k,v_1)$$ 
    for some walks $p_i$ where $i \in [k]$. In fact the $p_i$ are paths since, if a vertex is repeated along $p_i$ then, by Lemma \ref{lem:vertex<3}, this vertex is visited twice by $w$ so it is a cut vertex. Thus, there will be an end block in the block tree $B(\mathbf{w})$ which $p_i$ will travel around which is not accounted for by $B_1,\dots, B_k$. Note that a path $p_i$ could have length $0$ in which case $v_i=v_{i+1}$ and then $c_i$ and $c_{i+1}$ would share this cut vertex. Now define the even closed walks, 
    $$w_i:=(v_i,c_i,v_i,p_i,v_{i+1},c_{i+1},v_{i+1},-p_i,v_i)$$ 
    for $i \in [k-1]$ and 
    $$w_{k}:=(v_k,c_k,v_k,p_k,v_{1},c_{1},v_{1},-p_k,v_k)$$ 
    which correspond to elements $z_i \in \ker_\ZZ(A_G)$, for all $i \in [k]$. Letting $(m)_{sf}$ denote the square-free part of the monomial $m$, we set the signs of the $B_{w_i}$ such that $(E^+(w_i))_{\mathrm{sf}}|E^+(w)$ and $(E^-(w_i))_{\mathrm{sf}}|E^-(w)$ for all $i \in [k]$. We can do this since, even though edges in the paths $p_i$ are traversed twice in $w_i$ as opposed to once in $w$, both edges will always appear on the same side of the binomial (hence the need to consider the square-free parts of $E^+(w_i)$ and $E^-(w_i)$). This implies that $\operatorname{Supp}(z_i^+) \subseteq \operatorname{Supp}(z^+)$ and $\operatorname{Supp}(z_i^-) \subseteq \operatorname{Supp}(z^-)$. We also notice that the $w_i$ are all of the form: two odd cycles intersecting at one vertex (if the path $p_i$ has length $0$) or two vertex disjoint odd cycles joined by a walk (if the path $p_i$ has length $\ge 1$) and thus, by parts (\ref{item:circ2}) and (\ref{item:circ3}) of Theorem \ref{thm:villarealcircuits}, the $z_i$ are all circuits. By our assumption that $M$ distance-reduces all the circuits of $A_G$, we can find $\hat z_i \in M$, which distance-reduces $z_i$ for all $i \in [k]$. Since $\operatorname{Supp}(z_i^+) \subseteq \operatorname{Supp}(z^+)$ and $\operatorname{Supp}(z_i^-) \subseteq \operatorname{Supp}(z^-)$, we can use Lemma \ref{lem:cycleDR} on $\hat z_i,z_i,z$. Indeed, if $\hat z_i$ is a cycle for some $i \in [k]$, Lemma \ref{lem:cycleDR} tells us that $\hat z_i$ also distance-reduces $z$ and, since $\hat z_i \in M$, then $M$ distance-reduces $z$ and we have a contradiction. 
    
    Thus, we can assume that none of the $\hat z_i$ are cycles. We now use Theorem \ref{thm:tatakisCI}. All $\hat z_i$ are in our minimal Markov basis $M$ so they are all minimal and, since none of them are cycles, they must all correspond to the same element of $M$. However, since we ignore the signs of elements in our minimal Markov basis $M$, this implies that the $\hat z_i$ are only equal up to sign. Thus, we let $\hat z :=\hat z_1$ with corresponding walk $\hat w$. For $2$ elements $\hat y,y$, if $\hat y$ distance-reduces $y$ then $-\hat y$ also distance-reduces $y$. Thus, since $\hat z_i$ distance-reduces $z_i$, then $\hat z$, which is equal to either $\hat z_i$ or $-\hat z_i$, also distance-reduces $z_i$ for all $i \in [k]$. 
    
\vspace{2mm}
\textbf{Claim:} $k=2$.

\textit{Proof:} Suppose $k \neq 2$. Since $k \ge 2$, then $k \ge 3$. Now, fix some vertex $v \in V(\hat w)$. Since $\hat z$ distance-reduces $z_i$ for all $i \in [k]$, by Lemma \ref{lem:vertexin}, $v \in V(\hat w) \subseteq V(w_i)$ for all $i \in[k]$. In particular, $v$ is a vertex in at least $3$ walks $w_i$. For all $i \in [k]$, vertices in $ V(c_i) \backslash \{v_i\}$ appear exclusively in walks $w_{i-1}$ and $w_i$ since, if the vertices of $c_i$ intersected the vertices of $c_j$ for $j \neq i$ or $p_j$ for any $j \in [k]$ anywhere apart from at $v_i$, then this vertex $v'$ would be visited twice by $w$. Since $w$ is primitive, we can apply Lemma \ref{lem:vertex<3} which tells us that then $v'$ would be a cut vertex, contradicting the fact that the block $B_i$ has a unique cut vertex which is $v_i$. Now, since $v$ appears in at least $3$ walks $w_i$ then $v \not\in V(c_i) \backslash \{v_i\}$ for any $i \in [k]$ and so $v \in V(p_i)$ for all $i \in [k]$. This implies that $v$ is visited at least $3$ times by the walk $w$ which contradicts Lemma \ref{lem:vertex<3} and thus $k=2$.
\vspace{2mm}

Since $\hat z$ distance-reduces $z_1$, we can use Proposition \ref{prop:drredef} which tells us that $\hat z^+ \le z_1^+$ and $\operatorname{Supp}(\hat z^-) \cap \operatorname{Supp}(z_1^-) \neq \emptyset$ up to swapping signs of $\hat z$, $z_1$ and $z$. But, since $\operatorname{Supp}(z_1^-) \subseteq \operatorname{Supp}(z^-)$ then this implies that $\operatorname{Supp}(\hat z^-) \cap \operatorname{Supp}(z^-) \neq \emptyset$. However, $z$ can not be distance-reduced by $\hat z$ since $\hat z \in M$ and $M$ fails to distance-reduce $z$ and thus, by Proposition \ref{prop:drredef}, $\hat z^+ \not\le z^+$. Since $\hat w$ and $w$ are primitive, by condition (\ref{item:prim2}) of Theorem \ref{thm:reyesprimitive}, edges of $G$ are traversed at most twice by $\hat w$ and $w$ and so exponents in $E^+(\hat w)$ and $E^+(w)$ are at most $2$ and thus components of $\hat z,z$ have absolute value at most $2$. Also, $\operatorname{Supp}(\hat z^+) \subseteq \operatorname{Supp}(z_1^+) \subseteq\operatorname{Supp}(z^+)$ so the only way that $\hat z^+ \not\le z^+$ can happen is if there is some $e \in E(G)$ such that $\hat z_e=2$ and $z_e=1$. In terms of binomials, this means $e^2 | E^+( \hat w)$ and $e|E^+(w)$ but $e^2 \nmid E^+(w)$. We note that, since $\hat z^+ \le z_1^+$, then $E^+(\hat w)|E^+(w_1)$ so $e^2|E^+(w_1)$ as well, which implies $e \in E(p_1)$ since the only double edges of $w_1$ occur in $p_1$.


\vspace{2mm}
\textbf{Claim:} $e \not \in E(w_2)$.

\textit{Proof:} Suppose $e \in E(w_2)$. Then, either $e|E^+(w_2)$ or $e|E^-(w_2)$.
\begin{itemize}
    \item If $e|E^+(w_2)$, since $e$ does not appear along $c_1$ or $c_2$, then $e \in E(p_2)$ which implies $e^2 |E^+(w)$ since $e$ is traversed twice by $w$ (once in $p_1$ and again in $p_2$). This is a contradiction since $e^2 \nmid E^+(w)$ by construction.
    \item If $e|E^-(w_2)$, then, since $\operatorname{Supp}(z_2^-) \subseteq \operatorname{Supp}(z^-)$, then $e|E^-(w)$. But, we also have $e|E^+(w)$ and so $e|B_w$ which is a contradiction since the binomial $B_w$ is primitive since $w$ is primitive.
\end{itemize}

Thus, in both cases we come to a contradiction and thus $e \not\in E(w_2)$.
\vspace{2mm}

Since $\hat z$ also distance-reduces $z_2$ and the signs of $\hat z$ and $z_2$ are already fixed by the choices we made earlier, we have one of $\hat z^+ \le z_2^+$, $\hat z^+ \le z_2^-$, $\hat z^- \le z_2^+$ and $\hat z^- \le z_2^-$. If we are in one of the first two cases, then $E^+(\hat w)|E^+(w_2)$ or $E^+(\hat w)|E^-(w_2)$. Since $e|E^+(\hat w)$, then, either way, $e \in E(w_2)$, which is a contradiction to our previous claim. Furthermore, if $\hat z^- \le z_2^+$, then $\operatorname{Supp}(\hat z^-) \subseteq \operatorname{Supp}(z_2^+) \subseteq \operatorname{Supp}(z^+)$. But, $\operatorname{Supp}(\hat z^+) \subseteq \operatorname{Supp}(z^+)$ so $\operatorname{Supp}(\hat z) \subseteq \operatorname{Supp}(z^+)$ which is a contradiction since the edges corresponding to $z^+$ are disconnected. Thus, assume $\hat z^- \le z_2^-$. Since $\operatorname{Supp}(\hat z^+) \subseteq\operatorname{Supp}(z^+)$ but $z$ is not distance-reduced by $\hat z$ then $\hat z^- \not\le z^-$. Similarly to before, this implies there is some $f \in E(G)$ where $f|E^-(w)$ and $f^2|E^-(\hat w)$ but $f^2 \nmid E^-(w)$. In addition, $f \in E(w_2)$ but $f \not\in E(w_1)$, in particular, $f\not\in E(p_1)$.


Now, consider the two edges adjacent to $e$ in the walk $\hat w$, which we label $e_1,e_2$. Both $e_1,e_2$ divide $E^-(\hat w)$ and thus $E^{-}(w_2)$ since $\hat z^- \le z_2^-$. So $e_1,e_2 \in E(w_2)$. If we let $e=\{v_e,v'_e\}$, this implies $v_e,v'_e \in V(w_2)$ and, furthermore, since $e \in E(p_1)$, then $e \not\in E(c_1)$ and $e \not\in E(c_2)$ so $v_e,v'_e \in V(p_2)$. Since $p_2$ is a path, $v_e,v'_e$ are both visited exactly once by $p_2$. Thus, we can define the unique subwalk of $p_2$ between $v_e$ and $v'_e$ to be $p_e$. Note that $p_e$ is a path since $p_2$ is as well and $V(p_e)\cap V(p_1)=\{v_e,v'_e\}$ since otherwise the block of $\mathbf{w}$ containing $e$ is not a cycle or a cut edge so, by condition (\ref{item:prim1}) of Theorem \ref{thm:reyesprimitive}, $w$ is not primitive. Since $p_1$ is a path and $e=\{v_e,v'_e\} \in E(p_1)$, we can let $p'_1$ denote the path which is $p_1$ except, instead of travelling between $v_e$ and $v'_e$ using the edge $e$, we take the path $p_e$ instead. We note that neither $e$ nor $f$ are in $p'_1$. Now, consider the closed walk
$$w':=(v_1,c_1,v_1,p'_1,v_2,c_2,v_2,-p'_1,v_1).$$
Since $c_1,c_2$ are odd cycles then $w'$ is an even walk. We fix the sign of $B_{w'}$ such that $(E^+(w'))_{\mathrm{sf}}|E^+(w)$ and $(E^-(w'))_{\mathrm{sf}}|E^-(w)$. Now, let $z'$ denote the corresponding element in $\ker_\ZZ(A_G)$ and notice that $\operatorname{Supp}((z')^+) \subseteq \operatorname{Supp}(z^+)$ and $\operatorname{Supp}((z')^-) \subseteq \operatorname{Supp}(z^-)$. Furthermore, $z'$ is a circuit by part (\ref{item:circ3}) of Theorem \ref{thm:villarealcircuits} so it is distance-reduced by $M$. If the element in $M$ which distance-reduces $z'$ is not $\hat z$, then, by Theorem \ref{thm:tatakisCI}, it is a cycle. So, by Lemma \ref{lem:cycleDR}, $z$ is distance-reduced by this element and thus $M$ distance-reduces $z$ which is a contradiction. Thus, $z'$ is distance-reduced by $\hat z$ and, since $\hat z^+ \le z_1^+$ and $\hat z^- \le z_2^-$, we must have either $\hat z^+ \le (z')^+$ or $\hat z^- \le (z')^-$. If $\hat z^+ \le (z')^+$ then, since $e|E^+(\hat w)$ then $e|E^+(w')$ but this is a contradiction since $e \not\in E(w')$. If $\hat z^- \le (z')^-$ then, since $f|E^-(\hat w)$ then $f|E^-(w')$ which is also a contradiction since $f \not\in E(w')$ either.

This concludes the proof of the result.
\end{proof}

Theorem \ref{thm:CImain} is particularly interesting in light of Example \ref{ex:CIhombad} where it is shown that it is not enough just to check the circuits if $I_A$ is a complete intersection and homogeneous toric ideal. This demonstrates that the class of toric ideals coming from graphs is not generally as complicated as the class of homogeneous toric ideals, at least, not in the complete intersection case.

\section{Bouquet structure and distance reduction} \label{sec:Bouquets}

Let $A$ be an integer matrix. In \cite{petrovic2018bouquet}, the \textit{bouquets} $\{B_1, \dots, B_s\}$ of $A$, the \textit{bouquet matrix} $A_B$ and its associated \textit{bouquet ideal} $I_{A_B}$ are defined. These objects capture important combinatorial information about the toric ideal associated to $A$. Each bouquet is either free, non-mixed or mixed, which we label by $0,+$ and $-$ respectively and this is called the \textit{signature} of the bouquets. In addition, there is a bijection between $\ker_{\ZZ}(A)$ and $\ker_{\ZZ}(A_B)$ which preserves elements of the Graver basis and the circuits. In \cite{petrovic2018bouquet}, the authors also present an inverse construction where we begin with a matrix (which acts as our bouquet matrix) and a set of primitive vectors (which control the signature of the bouquets) and this produces a matrix, known as the \textit{generalised Lawrence matrix}, which has the desired bouquet ideal and signature. It is worth noting that, by \cite[Corollary~2.3]{petrovic2018bouquet}, every toric ideal arises as a generalised Lawrence matrix so this inverse construction produces all possible toric ideals which have the given bouquet ideal and signature.

In \cite{kosta2023strongly}, it is shown that two matrices which have the same bouquet ideal and signature have more in common than just their Graver basis and circuits and, in fact, there is a bijection between minimal Markov bases as well. When considering the distance-reduction property, this naturally prompts the question of whether the distance-reduction property is preserved between two such matrices.

In fact, the answer is no and we provide a counterexample now. 

\begin{example}
\label{ex: same bouquet diff DR}
Consider the two following matrices:
\[
U=\begin{bmatrix} u_1 & u_2 & u_3 & u_4 \end{bmatrix}=\begin{bmatrix} 1&1&0&1\\0&1&1&0 \end{bmatrix}
\text{ and }
V=\begin{bmatrix} v_1 & v_2 & v_3 & v_4 & v_5 \end{bmatrix}=\begin{bmatrix} 1&0&1&0&1 \\ 0&1&1&0&1 \\ 0&0&1&1&0 \end{bmatrix}.
\]
Bases for $\ker_\ZZ(U)$ and $\ker_\ZZ(V)$ are given by
$$\begin{bmatrix}1&1\\-1&0\\1&0\\0&-1\end{bmatrix} \quad \text{ and } \quad \begin{bmatrix}1&1\\1&1\\-1&0\\1&0\\0&-1\end{bmatrix}$$
so we can read off that both $U$ and $V$ have three bouquets,
\[
\{u_1\},\, \{u_2,u_3\},\, \{u_4\}
\quad \text{and} \quad
\{v_1,v_2\},\, \{v_3,v_4\},\, \{v_5\}
\]
and the signature for both sets of bouquets is  $+,-,+$. 
The bouquet-index-encoding vectors are then given by
$$\begin{pmatrix}1\\0\\0\\0\end{pmatrix},\begin{pmatrix}0\\1\\-1\\0\end{pmatrix},\begin{pmatrix}0\\0\\0\\1\end{pmatrix} \quad \text{ and } \quad \begin{pmatrix}1\\1\\0\\0\\0\end{pmatrix},\begin{pmatrix}0\\0\\1\\-1\\0\end{pmatrix},\begin{pmatrix}0\\0\\0\\0\\1\end{pmatrix}$$
and the bouquet ideal for both is $I_T$ where $T=\begin{bmatrix} 1 & 1 & 1 \end{bmatrix}$.

Since $I_U$ and $I_V$ have the same bouquet ideal and signature, by \cite{kosta2023strongly}, we have a bijection between the minimal Markov bases of $U$ and $V$. Indeed, both $U$ and $V$ have two minimal Markov bases, which we label  
$$M_U=\begin{bmatrix} 1&0&0&-1\\0&1&-1&-1\end{bmatrix} \quad M_U'=\begin{bmatrix} 1&0&0&-1\\1&-1&1&0\end{bmatrix}$$
$$M_V=\begin{bmatrix} 1&1&0&0&-1\\0&0&1&-1&-1\end{bmatrix} \quad M_V'=\begin{bmatrix} 1&1&0&0&-1\\1&1&-1&1&0\end{bmatrix}$$ 
respectively. We observe that $M_U,M_U',M_V$ are all distance-reducing whilst $M_V'$ reduces all elements of the Graver basis apart from $g=\begin{pmatrix} 0&0&1&-1&-1\end{pmatrix}$. This observation implies that the distance-reduction property is not preserved by the bijection between Markov bases. We confirm this computationally by running the code in the file \texttt{distanceReductionCode.m2} in \cite{githubM2DistanceReduction}
in Macaulay2 \cite{M2} and then using the \textit{mbfailures} command as follows.

\begin{lstlisting}
i1 : U = matrix "1,1,0,1;0,1,1,0"

o1 = | 1 1 0 1 |
     | 0 1 1 0 |

i2 : mbfailures U

o2 = {{| 1 0 0  -1 |, 0}, {| 1 0  0 -1 |, 0}}
       | 0 1 -1 -1 |       | 1 -1 1 0  |

i3 : V = matrix "1,0,1,0,1;0,1,1,0,1;0,0,1,1,0"

o3 = | 1 0 1 0 1 |
     | 0 1 1 0 1 |
     | 0 0 1 1 0 |

i4 : mbfailures V

o4 = {{| 1 1 0 0  -1 |, 0}, {| 1 1 0  0 -1 |, | 0 0 1 -1 -1 |}}
       | 0 0 1 -1 -1 |       | 1 1 -1 1 0  |
\end{lstlisting}

The command \textit{mbfailures} takes in a matrix and outputs all the minimal Markov bases alongside which elements of the Graver basis they fail to distance-reduce. Since the elements the minimal Markov bases fail to distance-reduce are formatted as matrices, there is just a $0$ next to a minimal Markov basis if all elements of the Graver basis are distance-reduced. By \cite[Theorem 7.4]{CK_2024}, this means that the minimal Markov basis in question is distance-reducing. Thus, from the Macaulay2 output we can see that $M_U,M_U',M_V$ are all distance-reducing whilst $M_V'$ is not.

We can examine why this difference occurs by looking at the fiber of $V$ containing $g^+$ and $g^-$. 
Since $V g^+=Vg^-=v_3$ then this fiber is $\mathcal{F}_{v_3}$. From Figure~\ref{fig:F111}, we see that we cannot reduce the distance between $g^+$ and $g^-$, which is $3$, by applying either move from $M_V'$. Indeed, if we apply $\begin{pmatrix}1&1&-1&1&0\end{pmatrix} \in M'_V$ to $g^+$, we move to $\begin{pmatrix}1&1&0&1&0\end{pmatrix} \in \mathcal{F}_{v_3}$ where the distance to $g^-$ stays at $3$ and doesn't decrease and if we apply $\begin{pmatrix}1&1&0&0&-1\end{pmatrix} \in M'_v$ to $g^-$ we also move to $\begin{pmatrix}1&1&0&1&0\end{pmatrix} \in \mathcal{F}_{v_3}$ which actually increases the distance to $g^+$ to $4$. Thus $g$ is not distance-reduced by $M'_V$ On the other hand, $g $ is an element of $ M_V$ so $g$ is trivially distance-reduced by $M_V$.

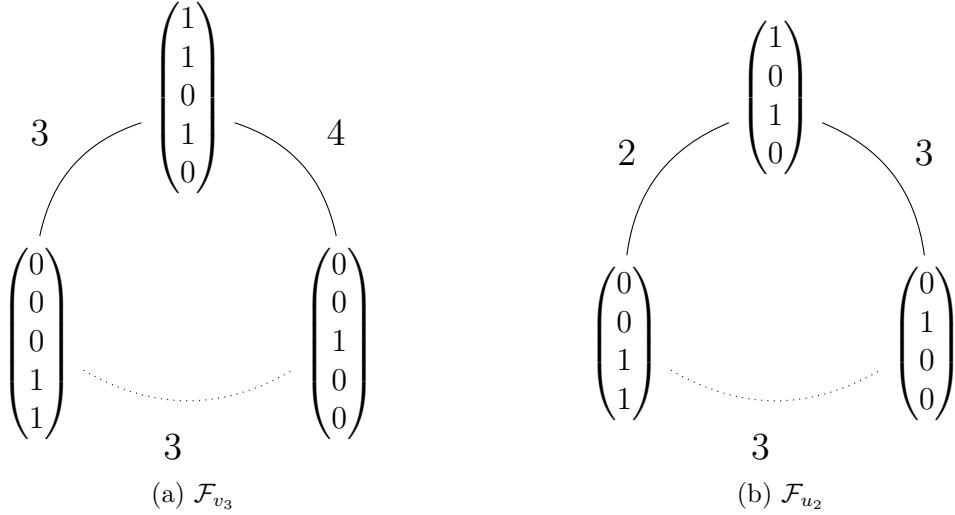
\begin{figure}
    \subfloat[\centering $\mathcal{F}_{v_3}$]{{
    \begin{tikzpicture} [scale=1]
	\begin{pgfonlayer}{nodelayer}
        \node[] (1) at (2,3.25) {$\begin{pmatrix} 1\\1\\0\\1\\0\end{pmatrix}$};
        \node[] (2) at (0,0) {$\begin{pmatrix} 0\\0\\0\\1\\1\end{pmatrix}$};
        \node[] (3) at (4,0) {$\begin{pmatrix} 0\\0\\1\\0\\0\end{pmatrix}$};
	\end{pgfonlayer}
	\begin{pgfonlayer}{edgelayer}
        \draw[bend right] (1) to (2) node [above=30] {\large $3$};
        \draw[bend right,dotted] (2) to (3) node [left=45,below=20] {\large $3$};
        \draw[bend left] (1) to (3) node [above=30] {\large $4$};
	\end{pgfonlayer}
    \end{tikzpicture}
    \label{fig:F111}
    }}%
    \hspace{2cm}
    \subfloat[\centering $\mathcal{F}_{u_2}$]{{
    \begin{tikzpicture} [scale=1]
	\begin{pgfonlayer}{nodelayer}
        \node[] (1) at (2,3.25) {$\begin{pmatrix} 1\\0\\1\\0\end{pmatrix}$};
        \node[] (2) at (0,0) {$\begin{pmatrix} 0\\0\\1\\1\end{pmatrix}$};
        \node[] (3) at (4,0) {$\begin{pmatrix} 0\\1\\0\\0\end{pmatrix}$};
	\end{pgfonlayer}
	\begin{pgfonlayer}{edgelayer}
        \draw[bend right] (1) to (2) node [above=30] {\large $2$};
        \draw[bend right,dotted] (2) to (3) node [left=45,below=20] {\large $3$};
        \draw[bend left] (1) to (3) node [above=30] {\large $3$};
	\end{pgfonlayer}
    \end{tikzpicture}
    \label{fig:F11}
    }}%
    \centering
    \caption{The fibers $\mathcal{F}_{v_3}$ and $\mathcal{F}_{u_2}$ with non-dotted edges corresponding to elements in $M_V'$ and $M'_U$ respectively and the $1$-norm distances between pairs of elements labelled.}
\end{figure}

To see what changes when we look at $M'_U$, we put together the bijections $\ker_\ZZ(U) \rightarrow \ker_\ZZ(T)$ and $\ker_\ZZ(V) \rightarrow \ker_\ZZ(T)$ to create a bijection $\ker_\ZZ(V) \rightarrow \ker_\ZZ(U)$ where $g \mapsto h=\begin{pmatrix} 0&1&-1&-1\end{pmatrix}$. Then, since $U h^+=U h^-=u_2$, $h$ lies in the fiber $\mathcal{F}_{u_2}$.
Indeed, if we apply the move $\begin{pmatrix}1&-1&1&0\end{pmatrix} \in M_U'$ to $h^+$, we move to $\begin{pmatrix}1&0&1&0\end{pmatrix} \in \mathcal{F}_{u_2}$. We see from Figure~\ref{fig:F11} that the distance to $h^-$ is then reduced from $3$ to $2$ and therefore $h$ is distance-reduced by $M'_U$. Also, as before, $h$ is an element of $M_U$ so $h$ is trivially distance-reduced by $M_U$.
\end{example}

\subsection{The poset of signatures} \label{sec:poset of signatures}

In the counterexample above, the distance-reduction information is not preserved between the two toric ideals, not because the structure of the fibers is different, but because elements in the fiber have different 1-norms. However, we can amend this, simply by requiring that the toric ideals are homogeneous and we prove this in Theorem \ref{thm:globaldr}.

We begin by recalling some more technicalities of the theory of bouquets from \cite{petrovic2018bouquet} and \cite{kosta2023strongly}, for completeness and also to set some notation. For a toric ideal $I_A$, we let $B_1,\dots,B_s$ denote the bouquets of $I_A$. Fixing a simple toric ideal $I_T \subseteq k[x_1, \dots ,x_s]$ for some integer matrix $T$ and a signature $\omega \subseteq [s]$, we recall the definition of a $T_{\omega}$-robust ideal. 

\begin{definition}\cite[Definition 2.1]{kosta2023strongly}
    Let $I_T$ be a simple toric ideal and $\omega \subseteq [s]$. A toric ideal $I_A$ is called $T_{\omega}$-robust if 
    \begin{itemize}
        \item the bouquet ideal of $I_A$ is $I_T$ and
        \item $\omega=\{i \in [s] \,|\, B_i\text{ is non-mixed}\}$.
    \end{itemize}
\end{definition}

Note that, if $I_A$ is a $T_{\omega}$-robust ideal, it is not necessarily the case that $T=A_B$. For example, when using the inverse construction with starting matrix $T$, we have that $T=A_B$ up to redundant rows of zeroes. 

\cite[Theorem~1.9]{petrovic2018bouquet} tells us that, letting $\mathbf{c}_i \in \ZZ^n$ be our bouquet-index-encoding vectors for $i \in[s]$, we have a bijection 
$$D:\ker_\ZZ(T) \rightarrow \ker_\ZZ(A),u \mapsto \mathbf{c}_1 u_1 +\dots+ \mathbf{c}_s u_s$$
which preserves the Graver basis and the circuits of $A$. Following the convention in \cite{petrovic2018bouquet}, we set the first non-zero coordinate of the bouquet-index-encoding vectors $\mathbf{c}_i$ to be positive.

By \cite[Remark 1.8]{petrovic2018bouquet}, if $A$ has a free bouquet, say $B_s$, then it is easy to see that
$$\ker_\ZZ(T)=\{(u_1,\dots,u_{s-1},0) \,|\, (u_1,\dots,u_{s-1})  \in \ker_\ZZ(T')\}$$
where $T'$ is the matrix which is the first $s$ columns of $T$. Thus, although there are infinitely many ways we could have defined $\mathbf{c}_s$, the sets $\ker_\ZZ(T)$ and thus $\ker_\ZZ(A)$ are independent of this choice. From here on, we assume $A$ does not have any free bouquets so if $i \not \in \omega$ then $B_i$ is mixed. In this case, each bouquet-index-encoding vector $\mathbf{c}_i$ is uniquely defined. This also allows us to identify signatures with subsets $\omega \subseteq [s]$ of non-mixed bouquets. For example, the signature $\begin{pmatrix}+&-&+\end{pmatrix}$ would correspond to $\omega=\{1,3\}$.

We will now fix a simple toric ideal $I_T$. Then, we take $\omega' \subseteq \omega \subseteq [s]$ and let $I_{A_{\omega}},I_{A_{\omega'}}$ be (not necessarily homogeneous) $T_{\omega},T_{\omega'}$-robust ideals respectively for some integer matrices $A_{\omega} \in \ZZ^{d \times n}$ and $A_{\omega'}\in \ZZ^{d' \times n'}$. For $i \in [s]$, we let $B_i$, $B'_i$ denote the bouquets of $A_{\omega}$, $A_{\omega'}$ and let $\mathbf{c}_i$, $\mathbf{c}'_i$ denote the bouquet-index-encoding vectors respectively. Also, we denote the bijections $D_{\omega}: \ker_\ZZ(T) \rightarrow \ker_\ZZ(A_{\omega})$ and $D_{\omega'}: \ker_\ZZ(T) \rightarrow \ker_\ZZ(A_{\omega'})$.

Now consider a poset containing all the signatures ordered by inclusion. Note that this poset is isomorphic to the Boolean lattice on $s$ elements. For example, taking $s=3$, since $\omega'=\{1\}$ is included in $\omega=\{1,3\}$, we think about the signature $\begin{pmatrix}+&-&-\end{pmatrix}$, corresponding to $\omega'$, as below the signature $\begin{pmatrix}+&-&+\end{pmatrix}$, which corresponds to $\omega$. The whole poset when $s=3$ is visualised in Figure \ref{fig:bool3}.


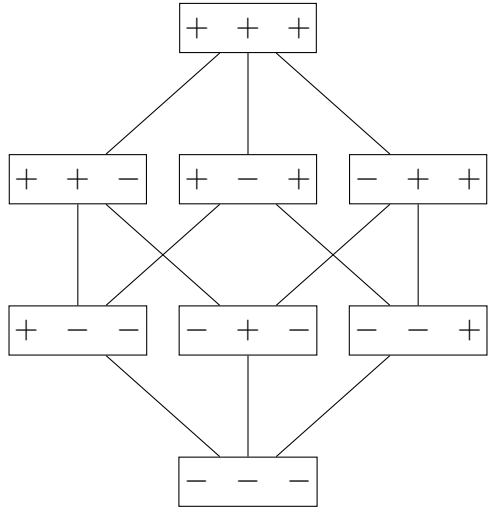
\begin{figure}[h!] 
    \begin{tikzpicture}[scale=1.5, vertices/.style={draw, fill=none, inner sep=2pt}]
              \node [vertices] (0) at (-0+0,0) {$\begin{matrix}-&-&-\end{matrix}$};
              \node [vertices] (1) at (-1.5+0,1.33333){$\begin{matrix}+&-&-\end{matrix}$};
              \node [vertices] (2) at (-1.5+1.5,1.33333){$\begin{matrix}-&+&-\end{matrix}$};
              \node [vertices] (4) at (-1.5+3,1.33333){$\begin{matrix}-&-&+\end{matrix}$};
              \node [vertices] (3) at (-1.5+0,2.66667){$\begin{matrix}+&+&-\end{matrix}$};
              \node [vertices] (5) at (-1.5+1.5,2.66667){$\begin{matrix}+&-&+\end{matrix}$};
              \node [vertices] (6) at (-1.5+3,2.66667){$\begin{matrix}-&+&+\end{matrix}$};
              \node [vertices] (7) at (-0+0,4){$\begin{matrix}+&+&+\end{matrix}$};
      \foreach \to/\from in {0/1, 2/3, 0/2, 1/3, 4/5, 6/7, 4/6, 5/7, 0/4, 1/5, 2/6, 3/7}
      \draw [-] (\to)--(\from);
      \end{tikzpicture}
    \centering
    \caption{This Figure displays the Hasse diagram of the poset of signatures when $s=3$.} 
    \label{fig:bool3}
\end{figure}

To each signature with corresponding set of non-mixed bouquets $\omega$, we associate the set of $T_{\omega}$-robust ideals. Thus, $A_{\omega}$ lies at the signature corresponding to $\omega$ and similarly with $A_{\omega'}$. In fact, even if $\omega$ and $\omega'$ are different, the kernels of $A_{\omega}$ and $A_{\omega'}$ are in bijection. Indeed, since $D_{\omega}$ is a bijection, we have the explicit bijection $D_{\omega'} \circ D_{\omega}^{-1}:\ker_\ZZ(A_{\omega}) \rightarrow \ker_\ZZ(A_{\omega'})$. Furthermore, by \cite[Theorem 2.2]{kosta2023strongly}, if we take two $T_{\omega}$-robust ideals $A_1$ and $A_2$, then $D_{\omega'} \circ D_{\omega}^{-1}$ induces a bijection between the minimal Markov bases of $A_1$ and $A_2$. However, this bijection does not necessarily hold between toric ideals lying at different locations in the poset i.e., there is generally no bijection between the minimal Markov bases of $A_{\omega}$ and $A_{\omega'}$ if $\omega \neq \omega'$. 

Since elements of the kernel are preserved anywhere in the poset, we can then ask what properties between those elements are preserved as we move around the poset. It turns out that some important properties are preserved if we move up and down the poset with respect to the order induced by inclusion. Recall that we let $\omega' \subseteq \omega$. Thus, moving from $\omega'$ to $\omega$ (from $\ker_\ZZ(A_{\omega'})$ to $\ker_\ZZ(A_{\omega})$) corresponds to moving \textbf{up} the poset and moving from $\omega$ to $\omega'$ (from $\ker_\ZZ(A_{\omega})$ to $\ker_\ZZ(A_{\omega'})$) corresponds to moving \textbf{down} the poset. 

Now, Lemma \ref{lem:globaldr1} will tell us that, for $\hat z,z \in \ker_\ZZ(A_{\omega'})$, if $\hat z^+ \le z^+$, then $(D_{\omega} \circ D_{\omega'}^{-1}(\hat z))^+ \le (D_{\omega} \circ D_{\omega'}^{-1}( z))^+$ ie. as we move \textbf{up} the poset, this relation on the positive parts of two vectors is preserved. In fact, this is one of the two conditions for $\hat z$ to distance-reduce $z$ by Proposition \ref{prop:drredef} so we note that this first condition holds as we move \textbf{up} the poset. Similarly, Lemma \ref{lem:globaldr2} will us that, if $A_\omega$ and $A_{\omega'}$ are homogeneous, for $\hat z,z \in \ker_\ZZ(A_{\omega})$, if $\operatorname{Supp}(\hat z^-) \cap \operatorname{Supp}(z^-) \neq \emptyset$ then $\operatorname{Supp}((D_{\omega'}\circ D_{\omega})(\hat z)^-) \cap \operatorname{Supp}((D_{\omega'}\circ D_{\omega})(z)^-) \neq \emptyset$ ie. as we move \textbf{down} the poset, the non-emptiness of the supports is preserved. This is the second condition for $\hat z$ to distance-reduce $z$ by Proposition \ref{prop:drredef} so, in opposition to the first condition, this second condition holds as we move \textbf{down} the poset. Although, the two conditions for distance reduction are, in some sense, working against each other, if we fix the signature ie. take two $T_{\omega}$-robust ideals $A_1$ and $A_2$ which are both homogeneous then, using Proposition~\ref{prop:drredef}, it follows quite quickly in Theorem \ref{thm:globaldr} that all homogeneous $T_{\omega}$-robust ideals have the same distance-reducing properties.

We will now prove Lemmas \ref{lem:globaldr1} and \ref{lem:globaldr2} and then Theorem \ref{thm:globaldr}.

\begin{remark}
    We will repeatedly use the fact that the index-encoding vectors have disjoint supports, i.e., the set $\{\operatorname{Supp}(\mathbf{c}_i) \,|\, i \in [s]\}$ is a partition of $[n]$. In particular, for each $k \in [n]$, there is a unique $i \in [s]$ such that $(\mathbf{c}_i)_k \neq 0$. Then unpacking the definition of $D_{\omega}$ for some $z \in \ker_\ZZ(T)$, 
\begin{equation} \label{eq:cidisj}
    (D(z))_k=\left(\sum_{j=1}^s z_j \mathbf{c}_j \right)_k=\sum_{j=1}^s z_j (\mathbf{c}_j)_k=z_i (\mathbf{c}_i)_k.
\end{equation}
We will use Equation (\ref{eq:cidisj}) so often we will not reference back to it.
\end{remark}

\begin{lemma} \label{lem:globaldr1}
    For $\hat y,y \in \ker_\ZZ(T)$, $D_{\omega'}(\hat y)^+ \le D_{\omega'}(y)^+$ if and only if $D_{\omega}(\hat y)^+ \le D_{\omega}(y)^+$ and, for every $i \in \operatorname{Supp}(\hat y^-) \backslash \omega'$, we have $y_i \le \hat y_i$.
\end{lemma}

\begin{proof}
    For the ($\Rightarrow$) direction, we assume $D_{\omega'}(\hat y)^+ \le D_{\omega'}(y)^+$ and start by showing that $D_{\omega}(\hat y)^+ \le D_{\omega}(y)^+$. Fix $k \in [n]$ and take the unique $i \in [n]$ such that $(\mathbf{c}_i)_k \neq 0$. It is enough to show that, if $D_{\omega}(\hat y)_k>0$, then $D_{\omega}(\hat y)_k \le D_{\omega}(y)_k$. So assume $D_{\omega}(\hat y)_k>0$. Then, $\hat y_i(\mathbf{c}_i)_k=D_{\omega}(\hat y)_k>0$ and we now split into two cases depending on the sign of $(\mathbf{c}_i)_k$.
    \begin{itemize}
        \item If $(\mathbf{c}_i)_k>0$ then $\hat y_i>0$. We can always find $k' \in [n']$ such that $(\mathbf{c}'_i)_{k'}>0$ and thus $D_{\omega'}(\hat y)_{k'}=\hat y_i ( \mathbf{c}'_i)_{k'}>0$. Since $D_{\omega'}(\hat y)^+ \le D_{\omega'}(y)^+$ then $\hat y_i(\mathbf{c}'_i)_{k'}=D_{\omega'}(\hat y)_{k'} \le D_{\omega'}(y)_{k'}=y_i (\mathbf{c}_i)_{k'}$. Since $(\mathbf{c}'_i)_{k'}>0$ then $\hat y_i \le y_i$ and, since$(\mathbf{c}_i)_k>0$, then $D_{\omega}(\hat y)_k=\hat y_i (\mathbf{c}_i)_k \le y_i (\mathbf{c}_i)_k= D_{\omega}(y)_k$.
        \item If $(\mathbf{c}_i)_k<0$ then $\hat y_i<0$ and $i \not \in \omega$. Since $\omega' \subseteq \omega$ then $i \not\in \omega'$ as well so $B'_i$ is mixed. Thus, we can find $k' \in [n']$ such that $(\mathbf{c}'_i)_{k'}<0$ and thus $D_{\omega'}(\hat y)_{k'}=\hat y_i ( \mathbf{c}'_i)_{k'}>0$. Since $D_{\omega'}(\hat y)^+ \le D_{\omega'}(y)^+$ then $\hat y_i(\mathbf{c}'_i)_{k'}=D_{\omega'}(\hat y)_{k'} \le D_{\omega'}(y)_{k'}=y_i (\mathbf{c}_i)_{k'}$. Since $(\mathbf{c}'_i)_{k'}<0$ then $\hat y_i \ge y_i$ and, since $(\mathbf{c}_i)_k<0$, then $D_{\omega}(\hat y)_k=\hat y_i (\mathbf{c}_i)_k \le y_i (\mathbf{c}_i)_k= D_{\omega}(y)_k$.
    \end{itemize}

Finally, we show that for some $ i \in \operatorname{Supp}(\hat y^-) \backslash \omega'$ we have $y_i \le \hat y_i$. Since $i \not\in \omega'$ then $B'_i$ is mixed and thus we can find $k' \in [n']$ such that $(\mathbf{c}'_i)_{k'}<0$. Also, since $i \in \operatorname{Supp}(\hat y^-)$ then $\hat y_i<0$ so $D_{\omega'}(\hat y)_{k'}=\hat y_i (\mathbf{c}'_i)_{k'}>0$. Since $D_{\omega'}(\hat y)^+ \le D_{\omega'}(y)^+$, then $\hat y_i(\mathbf{c}'_i)_{k'}=D_{\omega'}(\hat y)_{k'} \le D_{\omega'}(y)_{k'}=y_i (\mathbf{c}'_i)_{k'}$ and, since $(\mathbf{c}'_i)_{k'}<0$ then $y_i \le \hat y_i$.

For the ($\Leftarrow$) direction, we assume $D_{\omega}(\hat y)^+ \le D_{\omega}(y)^+$ and $\forall i \in \operatorname{Supp}(\hat y^-) \backslash \omega'$ we have $y_i \le \hat y_i$.  Fix $k' \in [n']$ and take the unique $i \in [s]$ such that $(\mathbf{c}'_i)_{k'} \neq 0$. Since we want to show that $D_{\omega'}(\hat y)^+ \le D_{\omega'}(y)^+$, it is enough to show that, if $D_{\omega'}(\hat y)_{k'}>0$ then $D_{\omega'}(\hat y)_{k'} \le D_{\omega'}(y)_{k'}$. So assume $D_{\omega'}(\hat y)_{k'}>0$. Then, $\hat y_i(\mathbf{c}'_i)_{k'}=D_{\omega'}(\hat y)_{k'}>0$ and we now split into two cases depending on the sign of $(\mathbf{c}'_i)_{k'}$.
    \begin{itemize}
        \item If $(\mathbf{c}'_i)_{k'}>0$ then $\hat y_i>0$. We can always find $k \in [n]$ such that $(\mathbf{c}_i)_{k}>0$ and thus $D_{\omega}(\hat y)_{k}=\hat y_i ( \mathbf{c}_i)_{k}>0$. Since $D_{\omega}(\hat y)^+ \le D_{\omega}(y)^+$ then $\hat y_i(\mathbf{c}_i)_{k}=D_{\omega}(\hat y)_{k} \le D_{\omega}(y)_{k}=y_i (\mathbf{c}_i)_{k}$. Since $(\mathbf{c}_i)_{k}>0$ then $\hat y_i \le y_i$ and, since$(\mathbf{c}'_i)_{k'}>0$, then $D_{\omega'}(\hat y)_{k'}=\hat y_i (\mathbf{c}'_i)_{k'} \le y_i (\mathbf{c}'_i)_{k'}= D_{\omega'}(y)_{k'}$.
        \item If $(\mathbf{c}'_i)_{k'}<0$ then $\hat y_i<0$ and $i \not \in \omega'$. Thus, $i \in \operatorname{Supp}(\hat y^-) \backslash \omega'$ and so, by assumption, we have $y_i \le \hat y_i$. Since$(\mathbf{c}'_i)_{k'}<0$, then $D_{\omega'}(\hat y)_{k'}=\hat y_i (\mathbf{c}'_i)_{k'} \le y_i (\mathbf{c}'_i)_{k'}= D_{\omega'}(y)_{k'}$.
    \end{itemize}
\end{proof}

\begin{remark}
    Lemma \ref{lem:globaldr1} implies immediately, that, if $D_{\omega'}(\hat y)^+ \le D_{\omega'}(y)^+$ then $D_{\omega}(\hat y)^+ \le D_{\omega}(y)^+$. In the language of semi-conformal decompositions, letting $\bar y=y-\hat y$, this means $D_{\omega'}(y)=D_{\omega'}(\hat y)+_{sc}D_{\omega'}(\bar y)$ $\Rightarrow$ $D_{\omega}(y)=D_{\omega}(\hat y)+_{sc}D_{\omega}(\bar y)$. Thus, we can now see that, if $D_{\omega}(y)$ is indispensable (has no semi-conformal decompositions), then $D_{\omega'}(y)$ is also indispensable. In the language of \cite[Section 3]{kosta2023strongly}, this is saying $D_{\omega}(y) \in S_{\omega}(T)$ implies $D_{\omega'}(y) \in S_{\omega'}(T)$. Thus we see that Lemma \ref{lem:globaldr1} implies \cite[Proposition 3.3]{kosta2023strongly}.
\end{remark}

\begin{remark}
    By \cite[Corollary 4.4]{petrovic2018bouquet}, all $T_{\emptyset}$-robust ideals are strongly robust, meaning that no element has a non-trivial semi-conformal decomposition. $T_{\emptyset}$-robust ideals lie at the bottom of the poset of signatures and this aligns with Lemma \ref{lem:globaldr1} which tells us that semi-conformal decompositions become harder and harder as we move down the poset until we get to the bottom and there are no non-trivial semi-conformal decompositions.
\end{remark}

\begin{corollary}\label{cor:globaldr1}
    For $\omega' \subseteq [s]$ and $\hat y, y \in \ker_\ZZ(T)$ then, if $D_{\omega'}(\hat y)^+ \le D_{\omega'}(y)$, then $\hat y^+ \le y^+$.
\end{corollary}

\begin{proof}
    We note that $I_T$ itself is a $T_{[n]}$-robust ideal and, in this case, $D_{[n]}$ is the identity. Thus, this result follows from the ($\Rightarrow$) direction of Lemma \ref{lem:globaldr1} taking $\omega=[s]$ and $A_{\omega}$ to be $T$.
\end{proof}

\begin{lemma} \label{lem:globaldr2}
    For $\hat y,y \in \ker_\ZZ(T)$, if $\operatorname{Supp}(D_{\omega}(\hat y)^-) \cap \operatorname{Supp}(D_{\omega}(y)^-) \neq \emptyset$ then $\operatorname{Supp}(D_{\omega'}(\hat y)^-) \cap \operatorname{Supp}(D_{\omega'}(y)^-) \neq \emptyset$.
\end{lemma}

\begin{proof}
    Since $\operatorname{Supp}(D_{\omega}(\hat y)^-) \cap \operatorname{Supp}(D_{\omega}(y)^-) \neq \emptyset$, we can find $k \in \operatorname{Supp}(D_{\omega}(\hat y)^-) \cap \operatorname{Supp}(D_{\omega}(y)^-)$ and take the unique $i \in [n]$ such that $(\mathbf{c}_i)_k \neq 0$. Now notice that $\hat y_i (\mathbf{c}_i)_k=D_{\omega}(\hat y)_k<0$ and $ y_i (\mathbf{c}_i)_k=D_{\omega}(y)_k<0$ and split into two cases depending on the sign of $(\mathbf{c}_i)_k$.
    \begin{itemize}
        \item If $(\mathbf{c}_i)_k>0$ then $\hat y_i<0$ and $y_i<0$. We can always find $k' \in [n]$ such that $(\mathbf{c}'_i)_{k'}>0$ and then $D_{\omega'}(\hat y)_{k'}=\hat y_i (\mathbf{c}'_i)_{k'}<0$ and $D_{\omega'}(y)_{k'}=y_i (\mathbf{c}'_i)_{k'}<0$. This means $k' \in \operatorname{Supp}(D_{\omega'}(\hat y)^-) \cap \operatorname{Supp}(D_{\omega'}(y)^-) $ and so $\operatorname{Supp}(D_{\omega'}(\hat y)^-) \cap \operatorname{Supp}(D_{\omega'}(y)^-) \neq \emptyset$.
        \item If $(\mathbf{c}_i)_k<0$ then $\hat y_i>0$ and $y_i>0$ and also $i \not\in \omega$. Since $\omega' \subseteq \omega$ then $i \not\in \omega'$ as well so $B'_i$ is mixed. Thus, we can find $k' \in [n]$ such that $(\mathbf{c}'_i)_{k'}<0$ and then $D_{\omega'}(\hat y)_{k'}=\hat y_i (\mathbf{c}'_i)_{k'}<0$ and $D_{\omega'}(y)_{k'}=y_i (\mathbf{c}'_i)_{k'}<0$. This means $k' \in \operatorname{Supp}(D_{\omega'}(\hat y)^-) \cap \operatorname{Supp}(D_{\omega'}(y)^-) $ and so $\operatorname{Supp}(D_{\omega'}(\hat y)^-) \cap \operatorname{Supp}(D_{\omega'}(y)^-) \neq \emptyset$.
    \end{itemize}
\end{proof}


Now let $I_{A_1},I_{A_2}$ be two homogeneous $T_{\omega}$-robust ideals for some $\omega \subseteq [s]$. As before just with slightly different notation, we define bijections $D_1:\ker_\ZZ(T) \rightarrow \ker_\ZZ(A_1)$ and $D_2:\ker_\ZZ(T) \rightarrow \ker_\ZZ(A_2)$. We can then define the natural bijection 
$$\phi:=D_2 \circ D_1^{-1}:\ker_\ZZ(A_1) \rightarrow \ker_\ZZ(A_2).$$
We have already seen that, by \cite[Theorem 1.11]{petrovic2018bouquet}, $\phi$ preserves the Graver basis and the circuits between $A_1$ and $A_2$ and, by \cite[Theorem 2.2]{kosta2023strongly}, $\phi$ also preserves the minimal Markov bases between $A_1$ and $A_2$. We now show that, with the homogeneity condition on $I_{A_1}$ and $I_{A_2}$ then $\phi$ also preserves the distance-reducing properties between $A_1$ and $A_2$. Since $\phi$ is a bijection, this means that $A_1$ and $A_2$ have exactly the same distance-reducing properties.

\begin{theorem}\label{thm:globaldr}
    Let $I_{A_1},I_{A_2}$ be homogeneous $T_{\omega}$-robust ideals for some $\omega \subseteq [s]$. Then, $I_{A_1}$ and $I_{A_2}$ have the same distance-reducing properties. Namely, for $\hat z,z \in \ker_\ZZ(A_1)$, $z$ is distance-reduced by $\hat z$ if and only if $\phi(z)$ is distance-reduced by $\phi(\hat z)$.
\end{theorem}

\begin{proof}
     It is enough to show that, if $z$ is distance-reduced by $\hat z$, then $\phi(z)$ is distance-reduced by $\phi(\hat z)$ since $A_1$ and $A_2$ are symmetric. This means we can swap $A_1,A_2$ and use $\phi^{-1} := D_1 \circ D_2^{-1}:\ker_\ZZ(A_2) \rightarrow \ker_\ZZ(A_1)$ to prove the other direction. Indeed, we see that, if $\phi(z)$ is distance-reduced by $\phi(\hat z)$, then $z=\phi^{-1} (\phi(z))$ is distance-reduced by $\hat z=\phi^{-1} (\phi(\hat z))$.

So assume $z$ is distance-reduced by $\hat z$. By Proposition \ref{prop:drredef}, then, up to sign, $\hat z^+ \le z^+$ and $\operatorname{Supp}(\hat z^-) \cap \operatorname{Supp}(z^-) \neq \emptyset$. Let $\hat y=D_1^{-1}(\hat z)$ and $y=D_1^{-1}(z)$ then notice that $D_2(\hat y)=D_2 \circ D_1^{-1}(\hat z)=\phi(\hat z)$ and $D_2(y)=D_2 \circ D_1^{-1}(z)=\phi(z)$. By Lemma \ref{lem:globaldr1} with $\omega'=\omega$, $A_{\omega'}=A_1$ and $A_{\omega}=A_2$, since $ D_1(\hat y)^+=\hat z^+ \le z^+=D_1(y)^+$ then $\phi(\hat z)^+=D_2(\hat y)^+ \le D_2(y)^+=\phi(z)^+$. In addition, by Lemma \ref{lem:globaldr2} again with $\omega'=\omega$, $A_{\omega'}=A_1$ and $A_{\omega}=A_2$, since $\operatorname{Supp}(\hat z^-) \cap \operatorname{Supp}(z^-)=\operatorname{Supp}(D_1(\hat y)^-) \cap \operatorname{Supp}(D_1(y)^-) \neq \emptyset$ then $\operatorname{Supp}(\phi(\hat z)^-) \cap \operatorname{Supp}(\phi(z)^-)=\operatorname{Supp}(D_2(\hat y)^-) \cap \operatorname{Supp}(D_2(y)^-) \neq \emptyset$. Finally, using Proposition \ref{prop:drredef} in the reverse direction, we see that $\phi(z)$ is distance-reduced by $\phi(\hat z)$.
\end{proof}

\begin{corollary}
    Let $M$ be a distance-reducing minimal Markov basis for $A_1$, then $\phi(M):=\{\phi(m) \,|\, m \in M\}$ is a distance-reducing minimal Markov basis for $A_2$.
\end{corollary}

\begin{proof}
    Theorem~2.2 from \cite{kosta2023strongly} tells us that, if $M$ is a minimal Markov basis of $A_1$, then $\phi(M)$ is a minimal Markov basis of $A_2$. Now, for some $w \in \ker_\ZZ(A)$, since $M$ is distance-reducing, there exists some $m \in M$ which distance-reduces $\phi^{-1}(w) \in \ker_\ZZ(A_1)$. Now, by Theorem \ref{thm:globaldr}, $w=\phi(\phi^{-1}(w))$ is distance-reduced by $\phi(m)$ where $\phi(m) \in \phi(M)$. Since $w$ was arbitrary then $\phi(M)$ is distance-reducing.
\end{proof}



\subsection{Homogeneity in terms of bouquets}



As outlined at the start of Section \ref{sec:Bouquets}, in \cite[Theorem 2.1]{petrovic2018bouquet}, an inverse construction is described. Given a matrix and a set of primitive vectors, this produces a generalised Lawrence matrix. In particular, for a simple toric ideal $I_T$ and a signature $\omega \subseteq [s]$, this allows us to construct infinitely-many $T_{\omega}$-robust ideals by letting $T$ be our matrix and ensuring that our $i$th primitive vector has a negative component if and only if $i \not\in \omega$. 

We will use this inverse construction in Examples \ref{ex:prehomrow}, \ref{ex:CIhombad}, \ref{ex:CIhombad2} and \ref{ex: A3}. However, based on Section \ref{sec:poset of signatures}, we only want to deal with homogeneous toric ideals $I_A$.
This motivates Theorem \ref{thm:homrow} in which we describe a necessary and sufficient condition on the bouquet matrix $A_B$ and bouquet-index-encoding vectors $\mathbf{c}_1,\dots ,\mathbf{c}_s$ such that the toric ideal $I_A$ is homogeneous. In particular, homogeneity of $I_A$ depends only on its bouquet data. In addition, since the starting matrix and primitive vectors of the generalised Lawrence matrix $A$ act as the bouquet matrix and bouquet-index-encoding vectors, this allows us to easily construct infinitely-many \textbf{homogeneous} $T_{\omega}$-robust ideals.  

As in the previous section, we let $B_1,\dots,B_s$ be the bouquets of $A$ with bouquet-index-encoding vectors $\mathbf{c}_i$ and bijection $D:\ker_\ZZ(A_B) \rightarrow \ker_\ZZ(A)$. Since $D$ is a linear bijection, we can extend it to the bijection $D_{\QQ}:\ker_\QQ(A_B) \rightarrow \ker_\QQ(A),u \mapsto \mathbf{c}_1 u_1+\dots +\mathbf{c}_s u_s$. For ease of notation, we let $\mathbf{\bar c}_i = \sum_{k=1}^{n} (\mathbf{c}_i)_k$ and $\mathbf{\bar c}=(\mathbf{\bar c}_1,\dots,\mathbf{\bar c}_s) \in \ZZ^s$ for all $i \in [s]$.

Before we prove Theorem~\ref{thm:homrow}, we provide an example of a generalised Lawrence matrix and how it interacts with homogeneity.

\begin{example} \label{ex:prehomrow}
    As input for the inverse construction, we take $T= \begin{bmatrix} 1 & 2 &0&-1 \\ 3 & -2 & 1&3 \end{bmatrix}$ and our integer vectors to be $\mathbf{c}_1=(3,1)$, $\mathbf{c}_2=(2,-5,3)$, $\mathbf{c}_3=(1)$, $\mathbf{c}_4=(5,-3)$. Strictly speaking, these are not bouquet-index-encoding vectors even though we use the same notation, $\mathbf{c}_i$. The true bouquet-index-encoding vectors will be obtained from these vectors by placing their entries in the coordinates corresponding to the relevant bouquet and zeroes elsewhere hence why we ignore this conflation of notation. 
    
    A basis for the integer kernel of $T$ is given by
    $$\begin{bmatrix}1&-1\\-1&0\\-2&6\\-1&-1\end{bmatrix}$$
    
    so, since none of the rows are multiples of each other, we can see that $I_T$ is a simple toric ideal since every bouquet of $T$ will be a singleton. Then, a generalised Lawrence matrix is
    $$A=\begin{bmatrix} 
    0 & 1 & -2 & 0 & 2 & 0 & 1 & 2 \\
    0 & 3 & 2 & 0 & -2 & 1 & -3 & -6 \\
    -1 & 3 & 0 & 0 & 0 & 0 & 0 & 0 \\
    0 & 0 & 5 & 2 & 0 & 0 & 0 & 0 \\
    0 & 0 & -3 & 0 & 2 & 0 & 0 & 0 \\
    0 & 0 & 0 & 0 & 0 & 0 & 3 & 5
    \end{bmatrix}.$$
Up to rows of zeroes, $T=A_B$ is the bouquet matrix of $A$ and we can also see that $\begin{pmatrix}1&1&-1&\frac{1}{2}&\frac{1}{2}&1 \end{pmatrix} A=\mathbf{1}_n$ and therefore $I_A$ is homogeneous. We also calculate that $\mathbf{\bar c}=\begin{pmatrix}
        \mathbf{\bar c}_1&\mathbf{\bar c}_2&\mathbf{\bar c}_3 \end{pmatrix}=\begin{pmatrix}4&0&1&2\end{pmatrix}$ and, if we add up the rows of $A_B$, we also get $\begin{pmatrix}
        4&0&1&2
    \end{pmatrix}$. Thus, while $I_{A_B}$ is not itself homogeneous, we notice that $\mathbf{\bar c} \in \operatorname{Rowspan}_\QQ(A_B)$. Theorem~\ref{thm:homrow} shows that this behaviour holds in general.
\end{example}

\begin{theorem} \label{thm:homrow}
    For an integer matrix $A$, $I_A$ is homogeneous if and only if $\mathbf{\bar c} \in \operatorname{Rowspan}_\QQ(A_B)$.
\end{theorem}

\begin{proof}
For the ($\Rightarrow$) direction, we suppose $I_A$ is homogeneous ie. $\mathbf{1}_n \in \operatorname{Rowspan}_\QQ(A)$. Since $\operatorname{Rowspan}_\QQ(A)=(\ker_\QQ(A))^{\perp}$, then $v \cdot \mathbf{1}_n=0$ for all $v \in \ker_\QQ(A)$. Now, for any $u \in \ker_\QQ(A_B)$, we note that $D_{\QQ}(u) \cdot \mathbf{1}_n = 0$. Thus,
\begin{align*}
    u \cdot \mathbf{\bar c}=\sum_{i=1}^s u_i \mathbf{\bar c}_i=\sum_{i=1}^s u_i (\mathbf{c}_i \cdot \mathbf{1}_n)=\left(\sum_{i=1}^s u_i \mathbf{c}_i \right) \cdot \mathbf{1}_n=D_{\QQ}(u) \cdot \mathbf{1}_n=0
\end{align*}
and so $\mathbf{\bar c} \in (\ker_\QQ(A_B))^{\perp}=\operatorname{Rowspan}_\QQ(A_B)$.

The ($\Leftarrow$) direction is effectively the same argument. Since $\mathbf{\bar c} \in \operatorname{Rowspan}_\QQ(A_B)=(\ker_\QQ(A_B))^{\perp}$ then, for any $v \in \ker_\QQ(A)$, we let $u:=D_{\QQ}^{-1}(v) \in \ker_\QQ(A_B)$. Then
\begin{align*}
    v \cdot \mathbf{1}_n=D_{\QQ}(u) \cdot \mathbf{1}_n=\left(\sum_{i=1}^s u_i \mathbf{c}_i \right) \cdot \mathbf{1}_n=\sum_{i=1}^s u_i (\mathbf{c}_i \cdot \mathbf{1}_n)=\sum_{i=1}^s u_i \mathbf{\bar c}_i=u \cdot \mathbf{\bar c}=0
\end{align*}
and so $\mathbf{1}_n \in (\ker_\QQ(A))^{\perp}=\operatorname{Rowspan}_\QQ(A)$ and thus $I_A$ is homogeneous.
\end{proof}


Theorem \ref{thm:homrow} provides further evidence that homogeneity seems to allow nice interactions between distance reduction and bouquets. However, we note the following counterexample to the statement: If $I_A$ is homogeneous and complete intersection, then a minimal Markov basis of $A$ is distance-reducing if and only if it distance-reduces the circuits of $A$. 


\begin{example} \label{ex:CIhombad}
    The following matrix 
    $$A=\begin{bmatrix}
-9 & 0 & 0 & 0 & 0 & 3 & 2 & 0 & 1 & 1 & 1 \\
-7 & 0 & 0 & 0 & 0 & 7 & 6 & 0 & 3 & 1 & 5 \\
-8 & 0 & 0 & 0 & 0 & 5 & 4 & 0 & 2 & 1 & 3 \\
-7 & 0 & 0 & 0 & 0 & 4 & 4 & 0 & 1 & 2 & 2 \\
4 & 1 & 0 & 0 & 0 & 0 & 0 & 0 & 0 & 0 & 0 \\
8 & 0 & 1 & 0 & 0 & 0 & 0 & 0 & 0 & 0 & 0 \\
-2 & 0 & 0 & 1 & 0 & 0 & 0 & 0 & 0 & 0 & 0 \\
0 & 0 & 0 & 0 & -1 & 2 & 0 & 0 & 0 & 0 & 0 \\
0 & 0 & 0 & 0 & 0 & 0 & -1 & 1 & 0 & 0 & 0 
\end{bmatrix}.$$
is a generalised Lawrence matrix starting with matrix
    $$
    T=\begin{bmatrix}
    -9&3&2&1&1&1\\
    -7&7&6&3&1&5\\
    -8&5&4&2&1&3\\
    -7&4&4&1&2&2
    \end{bmatrix},
    $$
and primitive vectors $\mathbf{c}_1=(1,-4,-8,2)$, $\mathbf{c}_2=(2,1)$, $\mathbf{c}_3=(1,1)$, $\mathbf{c}_4=\mathbf{c}_5=\mathbf{c}_6=(1)$. As before, none of the rows of any basis for the kernel of $T$ are multiples of each other so all the bouquets of $T$ are singletons and $I_T$ is a simple toric ideal. This implies that, up to rows of zeroes, $T=A_B$ is the bouquet matrix of $A$. Then, from Theorem \ref{thm:homrow}, we can immediately see that $I_A$ is homogeneous since $\mathbf{\bar c}=\begin{pmatrix}-9&3&2&1&1&1\end{pmatrix}$ which is the first row of $T$. We can also check that $I_A$ is homogeneous by definition. Taking  $\theta=\begin{bmatrix}0&-1&2&0&1&1&1&-1&1\end{bmatrix}$, we have $\theta A= \mathbf{1}_n$ and, therefore $\mathbf{1}_n \in \operatorname{Rowspan}(A)$.
    
    Using the \textit{AllMarkovBases} package \cite{CM_2025} for Macaulay2 \cite{M2}, we compute that $A$ has $5$ minimal Markov bases, each of which have $3$ elements, meaning that $I_A$ is complete intersection. Then, using the \textit{mbfailures} command \cite{githubM2DistanceReduction},
    as demonstrated in Example \ref{ex: same bouquet diff DR}, we see that all $5$ minimal Markov bases distance-reduce every element of the Graver basis apart from $v=\begin{pmatrix}1&-4&-8&2&10&5&-1&-1&1&0&-5 \end{pmatrix}$. Furthermore, $v$ is not a circuit since, for example, $c=\begin{pmatrix}1&-4&-8&2&8&4&0&0&3&0&-6\end{pmatrix} \in \ker_\ZZ(A)$ so $\operatorname{Supp}(c) \subset \operatorname{Supp}(v)$ meaning $v$ does not have minimal support. Thus, any minimal Markov basis, $M$, of $A$ distance-reduces all the circuits but does not distance-reduce $v$.

We now explain why $M$ fails to distance-reduce $v$ but does distance-reduce $c$. The choice of minimal Markov basis will not matter here so we choose
$$M=\begin{bmatrix}m_1\\m_2\\m_3\end{bmatrix}=\begin{bmatrix}
    0&0&0&0&0&0&1&1&0&-1&-1\\
    0&0&0&0&2&1&0&0&-2&-1&0\\
    1&-4&-8&2&0&0&0&0&11&4&-6
\end{bmatrix}.$$

To see why $v$ is not distance-reduced by $M$, we can use Proposition \ref{prop:drredef} since $I_A$ is homogeneous, where we switch signs as necessary. We can check exhaustively that the only options are $m_1^+ \le v^-$ and $m_2^+ \le v^+$ and we notice that $\operatorname{Supp}(m_1^-) \cap \operatorname{Supp}(v^+)=\{10,11\} \cap \{1,4,5,6,9\}=\emptyset$ and $\operatorname{Supp}(m_2^-) \cap \operatorname{Supp}(v^-)=\{9,10\} \cap \{2,3,7,8,11\}=\emptyset$. Thus, by Proposition \ref{prop:drredef}, $v$ is not distance-reduced by $m_1,m_2$ or $m_3$ so $v$ is not distance-reduced by $M$.

On the other hand, since $m_3^- =c^-$, in particular $m_3^- \le c^-$. Furthermore, $1 \in \operatorname{Supp}(m_3^+) \cap \operatorname{Supp}(c^+)$ so, by Proposition \ref{prop:drredef}, $c$ is distance-reduced by $m_3$ and thus by $M$. 
\end{example}

We present one more example of this type where, this time, $I_A$ is not complete intersection but $T$ is a monomial curve. Furthermore, even with $I_A$ being homogeneous, the universal Markov basis of $A$ (the union of the minimal Markov bases of $A$) distance-reduces the circuits but is not a distance-reducing set.


\begin{example} \label{ex:CIhombad2}
In example \ref{ex:CIhombad}, our simple toric ideal was not a monomial curve, however, we now present another example with $T=\begin{bmatrix}10&19&37&43&50\end{bmatrix}$. Since $T$ is a monomial curve then $I_T$ is a simple toric ideal so $T$ has $5$ bouquets. We let $A$ be the following homogeneous $T_{\omega}$-robust with signature $(+,+,-,+,+)$ i.e., $\omega=\{1,2,4,5\}$
$$A=\begin{bmatrix}
    10&0&19&0&0&-37&43&0&50&0\\
    -9&1&0&0&0&0&0&0&0&0\\
    0&0&-18&1&0&0&0&0&0&0\\
    0&0&0&0&1&38&0&0&0&0\\
    0&0&0&0&0&0&-42&1&0&0\\
    0&0&0&0&0&0&0&0&-49&1\\
\end{bmatrix}$$
which is a generalised Lawrence matrix with starting matrix $T$ and primitive vectors $\mathbf{c}_1=(1,9)$, $\mathbf{c}_2=(1,18)$, $\mathbf{c}_3=(38,-1)$, $\mathbf{c}_4=(1,42)$ and $\mathbf{c}_5=(1,49)$. Since $T$ is simple, then, up to rows of zeroes, $T=A_B$ is the bouquet matrix $A$. We have $\mathbf{\bar c}=\begin{pmatrix}10&19&37&43&50\end{pmatrix}$, which is clearly in the rowspan of $A_B$ so, by Theorem \ref{thm:homrow}, $A$ is homogeneous.

If we now take the minimal Markov basis 
$$M=\begin{bmatrix}
    3 & 27 & 2 & 36 & -114 & 3 & 1 & 42 & 0 & 0 \\
0 & 0 & 1 & 18 & 76 & -2 & -1 & -42 & -1 & -49 \\
2 & 18 & -3 & -54 & 38 & -1 & 0 & 0 & 0 & 0 \\
3 & 27 & 0 & 0 & -38 & 1 & -1 & -42 & 1 & 49 \\
5 & 45 & 0 & 0 & 0 & 0 & 0 & 0 & -1 & -49 \\
2 & 18 & -4 & -72 & -38 & 1 & 1 & 42 & 1 & 49 \\
1 & 9 & 4 & 72 & 0 & 0 & -2 & -84 & 0 & 0 \\
0 & 0 & 3 & 54 & 0 & 0 & 1 & 42 & -2 & -98 \\
9 & 81 & -7 & -126 & 0 & 0 & 1 & 42 & 0 & 0
\end{bmatrix}$$
then, using the \textit{mbfailures} command as in previous examples, we see that $M$ distance-reduces all elements of the Graver basis $\operatorname{Gr}(A)$ apart from the vectors
$$v=\begin{pmatrix}0&0&2&36&-76&2&2&84&-1&-49\end{pmatrix}$$ and $$w=\begin{pmatrix}10&90&-3&-54&0&0&-1&-42&0&0\end{pmatrix} \text{, }$$
neither of which is a circuit.
We can also see that the intersection of the elements which are not distance-reduced by each minimal Markov basis is $v$. Thus, the universal Markov basis of $A$ distance-reduces all elements of the Graver basis apart from $v$.
\end{example}

\subsection{Distance reduction in $\mathbb{A}^3$}\label{sec: distance reduction A3}

We now let $T$ be a monomial curve with positive integer entries. If $T$ is a monomial curve in $\mathbb{A}^1$, then $\ker_\ZZ(T)$ is empty so this is not an interesting case. Similarly, if $T$ is a monomial curve in $\mathbb{A}^2$, say $T=\begin{bmatrix}n_1 &n_2\end{bmatrix}$, then, letting $v:=\frac{1}{\gcd(n_1,n_2)}(n_2,-n_1) \in \ker_\ZZ(T)$, $\ker_\ZZ(T)=\langle v \rangle$. It is easy to see that $\{v\}$ is the unique minimal Markov basis of $T$ and $v$ distance-reduces everything so, in particular, this minimal Markov basis is distance-reducing.

When we let $T$ be a monomial curve in $\mathbb{A}^3$, much more complexity arises.  So, in the following we will assume that $T=\begin{bmatrix}n_1 & n_2 & n_3\end{bmatrix}$ where $n_1,n_2,n_3 \in \ZZ_{>0}$. Any matrix with bouquet ideal $I_T$ will have $3$ bouquets and, as before, we will represent the signature as some subset $\omega \subseteq [3]$ where $\omega$ contains indices corresponding to non-mixed bouquets.

\begin{remark} \label{rem:circuitsA3}
For $T \in  \ZZ_{>0}^3$, letting $g_{ij}=\gcd(n_i,n_j)$ for $i,j \in [3]$, the circuits of $T$ are, up to sign, $c_{12}:=\frac{1}{g_{12}}(n_2,-n_1,0)$, $c_{13}:=\frac{1}{g_{13}}(n_3,0,-n_1)$ and $c_{23}:=\frac{1}{g_{23}}(0,n_3,-n_2)$. In particular, we notice that they are supported on exactly two components and the positive part and negative are both supported on one component each.
\end{remark}

Since we are working with $T \in \ZZ^3_{>0}$, all elements of $\ker_\ZZ(T)$ have only $3$ components. In addition, since the $T_{\omega}$-robust ideal $A$ is homogeneous, we can exploit the support intersection condition from Proposition \ref{prop:drredef} to prove some strong results about distance reduction in this setting.
We will prove two propositions, Proposition \ref{prop:A3circuitDR} and Proposition \ref{prop:allcircuits} aiming towards Theorem \ref{thm:A3class} which gives the necessary and sufficient conditions of when a minimal Markov basis of a homogeneous $T_{\omega}$-robust matrix fails to be distance-reducing.

\begin{proposition} \label{prop:A3circuitDR}
    Let $T=\begin{bmatrix}
        n_1 & n_2 & n_3
    \end{bmatrix} \in  \ZZ_{>0}^3$, let $\omega \subseteq [3]$ and let $A$ be a homogeneous $T_{\omega}$-robust matrix. Then, for $\hat z,z \in \ker_\ZZ(A)$, if $\hat z^+ \le z^+$ but $z$ is not distance-reduced by $\hat z$, then $\hat z$ is a circuit.
\end{proposition}

\begin{proof}
    Letting $\hat y =D^{-1}(\hat z)$ and $y=D^{-1}(z)$, by Corollary \ref{cor:globaldr1} we note that $\hat y^+ \le y^+$. As noted in Remark \ref{rem:circuitsA3}, circuits of $T$ are precisely the elements of $\ker_\ZZ(T)$ where the positive and negative parts are supported on one component each. Suppose that $\hat y$ is not a circuit. Then,  since $A$ is positively-graded, $\ker_\ZZ(A) \cap \NN^n$ is empty so $\hat y^+ \neq 0$ and $\hat y^- \neq 0$. Thus, either $\hat y^+$ or $\hat y^-$ is supported on at least two components and we split into these two cases now. 

    \begin{itemize}
        \item For some $i,j \in \{1,2,3\}$ where $i \neq j$, suppose $i,j \in \operatorname{Supp}(\hat y^+)$  ie. $\hat y_i>0$ and $\hat y_j>0$. Since $\hat y^+ \le y^+$ then $i,j \in \operatorname{Supp}(y^+)$. Let $k$ be the element of $\{1,2,3\}$ which is not $i$ or $j$. Since $\hat y,y \in \ker_\ZZ(T)$ where $n_1,n_2,n_3 >0$ then $\hat y_k=-\frac{1}{n_k}(n_i \hat y_i + n_j \hat y_j)<0$ and, similarly, $y_k<0$. We can always find some $k' \in [n]$ such that $(\mathbf{c}_k)_{k'}>0$ and thus $\hat z_{k'}=D(\hat y)_{k'}=\hat y_k (\mathbf{c}_k)_{k'}<0$ and, similarly, $z_{k'}<0$ so $k' \in \operatorname{Supp}(\hat z^-) \cap \operatorname{Supp}(z^-)$. 
        \item Now, for some $i,j \in \{1,2,3\}$ where $i \neq j$, suppose $i,j \in \operatorname{Supp}(\hat y^-)$  ie. $\hat y_i<0$ and $\hat y_j<0$ and let $k$ be the element of $\{1,2,3\}$ which is not $i$ or $j$. Since $\hat y,y \in \ker_\ZZ(T)$ where $n_1,n_2,n_3 >0$ then $\hat y_k=-\frac{1}{n_k}(n_i \hat y_i + n_j \hat y_j)>0$. Thus, since $\hat y^+ \le y^+$, we have $0< \hat y_k<y_k$. If both $y_i\ge 0$ and $y_j\ge 0$ then, since $y \in \ker_\ZZ(T)$, $0=Ty=n_iy_i+n_jy_j+n_ky_k>0$ which is a contradiction. Thus, either $y_i<0$ or $y_j<0$, up to swapping $i$ and $j$, let $y_i<0$. We can always find some $i' \in [n]$ such that $(\mathbf{c}_i)_{i'}>0$ and thus $\hat z_{i'}=D(\hat y)_{i'}=\hat y_i (\mathbf{c}_i)_{i'}<0$ and, similarly, $z_{i'}<0$ so $i' \in \operatorname{Supp}(\hat z^-) \cap \operatorname{Supp}(z^-)$.
    \end{itemize}

    In either case, $\operatorname{Supp}(\hat z^-) \cap \operatorname{Supp}(z^-)$ is not empty. Since $A$ is homogeneous and $\hat z^+ \le z^+$, then, by Proposition \ref{prop:drredef}, $\hat z$ is distance-reduced by $z$ which is a contradiction since $z$ is not distance-reduced by $z$. Thus, $\hat y$ must be a circuit and, by \cite[Theorem 1.11]{petrovic2018bouquet}, $D$ preserves circuits, so, if $\hat y$ is a circuit then $\hat z$ is a circuit.
    
\end{proof}


\begin{proposition} \label{prop:allcircuits}
    Let $T=\begin{bmatrix}
        n_1 & n_2 & n_3
    \end{bmatrix} \in  \ZZ_{>0}^3$, let $\omega \subseteq [3]$ and let $A$ be a homogeneous $T_{\omega}$-robust matrix. For $z_1,z_2,z \in \operatorname{Gr}(A)$, if $z_1^+ \le z^+$ and $z_2^- \le z^-$ but $z_1$ and $z_2$ do not distance-reduce $z$, then the circuits of $A$ are $\mathcal{C}(A)=\{z_1,z_2,z\}$.
\end{proposition}

\begin{proof}
    By Proposition \ref{prop:A3circuitDR}, $z_1$ is a circuit and, since $(-z_2)^+=z_2^- \le z^-=(-z)^+$ and distance reduction does not depend on the signs of the vectors, Proposition \ref{prop:A3circuitDR} also implies that $z_2$ is a circuit. Letting $y=D^{-1}(z)$, $y_1=D^{-1}(z_1)$ and $y_2=D^{-1}(z_2)$, by Corollary \ref{cor:globaldr1}, $y_1^+ \le y^+$ and $y_2^- \le y^-$. Furthermore, by \cite[Theorem 1.11]{petrovic2018bouquet}, $D$ is a bijective correspondence between $\mathcal{C}(T)$ and $\mathcal{C}(A)$ so $y_1$ and $y_2$ are circuits and $|\mathcal{C}(A)|=3$. As discussed in Remark \ref{rem:circuitsA3}, the positive and negative parts of a circuit of $T$ are supported on exactly one component each, so let $\{i\}=\operatorname{Supp}(y_1^+)$, $\{j\}=\operatorname{Supp}(y_1^-)$ and $\{k\}=\operatorname{Supp}(y_2^-)$ for some $i,j,k \in [3]$. Since the supports of the positive and negative parts of a vector are disjoint then $\operatorname{Supp}(y_1^+) \cap \operatorname{Supp}(y_1^-)$ is empty so $i \neq j$. Using this observation again, since $y_1^+ \le y^+$ and $y_2^- \le y^-$, then $\operatorname{Supp}(y_1^+) \subseteq \operatorname{Supp}(y^+)$ and $\operatorname{Supp}(y_2^-) \subseteq  \operatorname{Supp}(y^-)$ and thus $\operatorname{Supp}(y_1^+) \cap \operatorname{Supp}(y_2^-) \subseteq \operatorname{Supp}(y^+) \cap \operatorname{Supp}(y^-)$ is empty so $i \neq k$. Suppose $j=k$. Then, since every bouquet-index-encoding vector has a positive component, we can find $j' \in [n]$ such that $(\mathbf{c}_j)_{j'}>0$. Thus $(z_1)_{j'}=D(y_1)_{j'}=(y_1)_j (\mathbf{c}_j)_{j'}<0$. Also, since $y_2^- \le y^-$, then $j=k\in \operatorname{Supp}(y_2^-) \subseteq \operatorname{Supp}(y^-)$ so $(z)_{j'}=D(y)_{j'}=(y)_j (\mathbf{c}_j)_{j'}<0$. Thus, $j' \in \operatorname{Supp}(z_1^-) \cap \operatorname{Supp}(z^-)$ and, combining this with $z_1^+ \le z^+$, Proposition \ref{prop:drredef} tells us that $z_1$ distance-reduces $z$. This is a contradiction since we know $z_1$ doesn't distance-reduce $z$ and thus $j \neq k$ and $\{i,j,k\}=[3]$. But now let $\{l\}=\operatorname{Supp}(y_2^+)$ for some $l \in [3]$. Using symmetric arguments, we have $l \neq k$ and $l \neq i$ so $l=j$. To summarise, we have $\{i\}=\operatorname{Supp}(y_1^+)$, $\{j\}=\operatorname{Supp}(y_1^-)=\operatorname{Supp}(y_2^+)$ and $\{k\}=\operatorname{Supp}(y_2^-)$ where $\{i,j,k\} =[3]$.

    We now claim that $y_j=0$. If $y_j<0$ then $(z)_{j'}=D(y)_{j'}=(y)_j (\mathbf{c}_j)_{j'}<0$. Also, since $j \in \operatorname{Supp}(y_1^-)$ then $(z_1)_{j'}=D(y_1)_{j'}=(y_1)_j (\mathbf{c}_j)_{j'}<0$. Thus, $j' \in \operatorname{Supp}(z_1^-) \cap \operatorname{Supp}(z^-)$ and, combining this with $z_1^+ \le z^+$, Proposition \ref{prop:drredef} tells us that $z_1$ distance-reduces $z$. However, this is a contradiction since $z_1$ does not distance-reduce $z$. A symmetric argument follows if $y_j>0$ so therefore $y_j=0$. Since $y$ is now only supported on two components, as discussed in Remark \ref{rem:circuitsA3}, $y$ is the third and final circuit of $T$, and, as before, $D$ preserves circuits, so $z=D(y)$ is the third and final circuit of $A$ and $\mathcal{C}(A)=\{z_1,z_2,z\}$.
\end{proof}


We now present a complete classification of when a minimal Markov basis of a homogeneous $T_{\omega}$-robust matrix fails to be distance-reducing for $T$ a monomial curve in $\mathbb{A}^3$.

\begin{theorem} \label{thm:A3class}
    Let $T=\begin{bmatrix}
        n_1 & n_2 & n_3
    \end{bmatrix} \in  \ZZ_{>0}^3$, $\omega \subseteq [3]$ and let $A$ be a homogeneous $T_{\omega}$-robust matrix with a minimal Markov basis $M$. Then, $M$ fails to be distance-reducing if and only if the following three conditions hold:
    \begin{itemize}
        \item $\omega=[3]$,
        \item $T$ is a complete intersection, and
        \item  $M$ contains two elements which are both circuits.
    \end{itemize}
In this case, $M$ distance-reduces all elements of $\operatorname{Gr}(A)$ except the only circuit of $A$ not contained in $M$.
\end{theorem}

\begin{proof}
    Suppose that $M$ is not distance-reducing. By \cite[Theorem 7.4]{CK_2024}, $M$ does not distance-reduce the Graver basis $\operatorname{Gr}(A)$ of $A$, so we have $z \in \operatorname{Gr}(A)$ such that $M$ does not distance-reduce $z$. Since $M$ is a minimal Markov basis, its moves connects $z^+$ and $z^-$ in the fiber, and thus we can find $z_1 \in M$ and $z_2 \in M$ such that $z_1^+ \le z^+$ and $z_2^- \le z^-$. Since $M$ does not distance-reduce $z$, then $z_1$ and $z_2^-$ do not distance-reduce $z$ so, by Proposition \ref{prop:allcircuits}, $\mathcal{C}(A)=\{z_1,z_2,z\}$. Letting $y=D^{-1}(z)$, $y_1=D^{-1}(z_1)$ and $y_2=D^{-1}(z_2)$, as in the proof of Proposition \ref{prop:allcircuits}, we let $\{i\}=\operatorname{Supp}(y_1^+)=\operatorname{Supp}(y^+)$, $\{j\}=\operatorname{Supp}(y_1^-)=\operatorname{Supp}(y_2^+)$ and $\{k\}=\operatorname{Supp}(y_2^-)=\operatorname{Supp}(y^-)$ where $\{i,j,k\} =[3]$.

    Now, suppose $\omega \neq [3]$, thus, at least one of $B_i,B_j,B_k$ is mixed. If $B_i$ is mixed, then we can find $i' \in [n]$ such that $(\mathbf{c}_i)_{i'}<0$. Thus, $(z_1)_{i'}=D(y_1)_{i'}=(y_1)_i (\mathbf{c}_i)_{i'}<0$ and $z_{i'}=D(y)_{i'}=y_i (\mathbf{c}_i)_{i'}<0$ so $i' \in \operatorname{Supp}(z_1^-) \cap \operatorname{Supp}(z^-)$. But, since $z_1^+ \le z^+$, then, by Proposition \ref{prop:drredef}, $z$ is distance-reduced by $z_1$ which is a contradiction. A symmetric argument holds for $B_k$, so therefore $B_j$ must be mixed. However, $j \in \operatorname{Supp}(y_1^-)$ so $(y_1)_j<0$ and, since $B_j$ is mixed, we can find some $j' \in [n]$ such that $(\mathbf{c}_j)_{j'}<0$. Thus, $(z_1)_{j'}=D(y_1)_{j'}=(y_1)_j (\mathbf{c}_j)_{j'}>0$ and, since $z_1^+ \le z^+$, then $z_{j'} \ge (z_1)_{j'}>0$. This contradicts the fact that $j \not\in \operatorname{Supp}(y^+) \cup \operatorname{Supp}(y^-)= \operatorname{Supp}(y)$ so $y_j=0$ and thus $(z)_{j'}=D(y)_{j'}=y_j (\mathbf{c}_j)_{j'}=0$. Since we have a contradiction then $\omega =[3]$. 
    
    Finally, $T$ and $A$ have the same bouquet and the same signature, so, by \cite[Theorem 2.2]{kosta2023strongly}, $D$ is a bijective correspondence between the minimal Markov bases of $T$ and the minimal Markov bases of $A$. The minimal Markov bases of monomial curves in $\mathbb{A}^3$ are very well-understood, having been studied by Herzog in \cite{Herzog1970} and \cite[Theorem 3.2]{CK_2024} gives a succinct description. If $T$ is not a complete intersection, then it has a unique minimal Markov basis which does not contain any circuits. This would imply that $A$ also has a unique minimal Markov basis which does not contain any circuits, which is a contradiction since $z_1 \in M$ and $z_2 \in M$ which are both circuits. Thus $T$ is a complete intersection and $M$ contains two elements, which are $z_1$ and $z_2$. As observed previously, if $M$ fails to distance-reduce some $z \in \operatorname{Gr}(A)$, then $z$ is the third and only circuit not in $M$. Indeed, since $z_1^+ \le z^+$ and $z_2^- \le z^-$ combined with the fact that $\operatorname{Supp}(z_1^-) \cap \operatorname{Supp}(z^-)$ and $\operatorname{Supp}(z_2^+) \cap \operatorname{Supp}(z^+)$ are both empty, Proposition \ref{prop:drredef} confirms that $z$ is not distance-reduced by either $z_1$ or $z_2$ so $M$ does not distance-reduce $z$.
\end{proof}

Contained within the proof of Theorem \ref{thm:A3class} is the following corollary.

\begin{corollary}
     Let $T=\begin{bmatrix}
        n_1 & n_2 & n_3
    \end{bmatrix} \in \ZZ_{>0}^3$ be a non-complete intersection, let $\omega \subseteq [3]$ and let $A$ be a homogeneous $T_{\omega}$-robust matrix. Then, all minimal Markov bases of $A$ are distance-reducing.
\end{corollary}

We also see that our circuits property holds in this case as well.

\begin{corollary}
     Let $T=\begin{bmatrix}
        n_1 & n_2 & n_3
    \end{bmatrix} \in \ZZ_{>0}^3$, $\omega \subseteq [3]$ and let $A$ be a homogeneous $T_{\omega}$-robust matrix with a minimal Markov basis $M$. Then $M$ is distance-reducing if and only if $M$ distance-reduces the circuits of $A$.
\end{corollary}

We now present an example.

\begin{example} \label{ex: A3}
    We will take $T=\begin{bmatrix}1&2&3\end{bmatrix}$. By Theorem \ref{thm:A3class}, if we pick any signature which is not $\begin{pmatrix}+&+&+\end{pmatrix}$ i.e., $\omega \neq \{1,2,3\}$, then any minimal Markov basis for a homogeneous $T_{\omega}$-robust matrix will be distance-reducing.

    Let us pick $\omega=\{2\}$ so that the signature will be $\begin{pmatrix}-&+&-\end{pmatrix}$. We can then construct a homogeneous $T_{\{2\}}$-robust ideal using the inverse construction from \cite[Section 2]{petrovic2018bouquet} and Theorem \ref{thm:homrow}. For the inverse construction, we take $T$ to be our starting matrix and let $\mathbf{c}_1=(2,-1)$, $\mathbf{c}_2=(1,1)$ and $\mathbf{c}_3=(4,-1)$ be our primitive vectors, giving us the generalised Lawrence matrix  
    $$A=\begin{bmatrix}
        0&-1&2&0&0&-3 \\
        1&2&0&0&0&0\\
        0&0&-1&1&0&0\\
        0&0&0&0&1&4
    \end{bmatrix}.$$
    We ensured that the signature of $A$ is $(-,+,-)$ since $\mathbf{c}_1$ and $\mathbf{c}_3$ have negative components while $\mathbf{c}_2$ doesn't. Also, since $I_T$ is simple, then, up to rows of zeroes, we have $A_B=T$. Finally, we picked our primitive vectors such that $\mathbf{c}=\begin{pmatrix}\mathbf{c}_1&\mathbf{c}_2&\mathbf{c}_3\end{pmatrix}=\begin{pmatrix}1&2&3\end{pmatrix}$. Thus, by Theorem \ref{thm:homrow}, $\mathbf{c}$ is in the rowspan of $T$ and hence $A_B$ and this implies that $I_A$ is homogeneous.
    
    
    
    Using the \textit{AllMarkovBases} package \cite{CM_2025} for Macaulay2 \cite{M2}, we compute that $A$ has a unique minimal Markov basis
    $$M=\begin{bmatrix}m_1\\m_2\\m_3\\m_4\end{bmatrix}=\begin{bmatrix}
        2&-1&1&1&-4&1\\
        0&0&3&3&-8&2\\
        2&-1&-2&-2&4&-1\\
        4&-2&-1&-1&0&0
    \end{bmatrix}.$$
    The only element which is in the Graver basis but not in $M$ is $v=\begin{bmatrix}6&-3&0&0&-4&1\end{bmatrix}$. Using Proposition \ref{prop:drredef}, we see that $v$ is distance-reduced by both $m_1$ and $m_4$. Checking this explicitly for $m_1$, we see that $m_1^- \le v^-$ and $\operatorname{Supp}(m_1^+) \cap \operatorname{Supp}(v^+)=\{1,3,4,6\} \cap \{1,6\}=\{1,6\}$ which is non-empty. In addition, $M$ trivially distance-reduces all of its own elements. By \cite[Theorem 7.4]{CK_2024}, it is enough to check if $M$ distance-reduces all the elements of the Graver basis to check if it is distance-reducing. Thus, $M$ is distance-reducing, which aligns with Theorem \ref{thm:A3class}.

    For the case when $\omega=[3]$ i.e., the signature is $(+,+,+)$, we can take
    $$A=\begin{bmatrix}
        1&2&0&3&0 \\
        0&-1&1&0&0\\
        0&0&0&-2&1
    \end{bmatrix}.$$
The matrix $A$ is a generalised Lawrence matrix constructed from the matrix $T$ and the primitive vectors $\mathbf{c}_1=(1)$, $\mathbf{c}_2=(1,1)$ and $\mathbf{c}_3=(1,2)$. Using Theorem \ref{thm:homrow} and similar checks as before, we can confirm that $A$ is indeed a homogeneous $T_{[3]}$-robust ideal. 
    
    This time $A$ has two minimal Markov bases,
  $$M_1=\begin{bmatrix}m\\m_1\end{bmatrix}=\begin{bmatrix}
        2&-1&-1&0&0\\
        1&1&1&-1&-2
    \end{bmatrix} \quad \text{and} \quad M_2=\begin{bmatrix}m\\m_2\end{bmatrix}=\begin{bmatrix}
        2&-1&-1&0&0\\
        3&0&0&-1&-2
    \end{bmatrix}.$$
The only elements in the Graver basis not in $M_1$ or $M_2$ are $v_1=\begin{bmatrix}1&-2&-2&1&2\end{bmatrix}$ and $v_2=\begin{bmatrix}0&3&3&-2&-4\end{bmatrix}$. Using Proposition \ref{prop:drredef}, we can see that $m_2$, $v_1$ and $v_2$ are all distance-reduced by $m_1$ so $M_1$ is distance-reducing. This aligns with Theorem~\ref{thm:A3class} since $m_1$ is not a circuit so $M$ does not contain two circuits and thus the third condition of Theorem~\ref{thm:A3class} is not fulfilled. However, since $m$ and $m_2$ are both circuits of $A$, $M_2$ does fulfil all $3$ conditions of Theorem \ref{thm:A3class} so Theorem \ref{thm:A3class} predicts that $M_2$ will not be distance-reducing and it will fail to distance-reduce the only circuit of $A$ not in $M_2$, which is $v_2$. Indeed, using Proposition \ref{prop:drredef}, $m_1$ is distance-reduced by $m_2$ and $v_1$ is distance-reduced by $m$. However, neither $m$ nor $m_2$ distance-reduce $v_2$ and thus $M_2$ distance-reduces all elements of the Graver basis except $v_2$.
\end{example}

In Theorem \ref{thm:A3class}, we managed to extend the distance-reduction properties of $T$ to $A$ leveraging the fact that $T$ was a monomial curve in $\mathbb{A}^3$. However, if we let $T$ be a monomial curve in $\mathbb{A}^n$ and $I_A$ be a homogeneous $T_{\omega}$-robust ideal, extending the distance-reduction properties of $T$ to $A$ is much more difficult. For example, we consider the property that, a minimal Markov basis is distance-reducing if and only if it distance-reduces the circuits. Then we can ask if there is a relation between this property holding for $T$ and this property holding for $A$.

For $\omega \neq [s]$, it seems unlikely there will be a relation since $T$ and $A$ might have different numbers of minimal Markov bases. But, for $\omega=[s]$, by \cite[Theorem 2.2]{kosta2023strongly}, we have a bijection between the minimal Markov bases of $T$ and $A$. However, we now present an example where the property does not hold for $T$ but does hold for $A$.

\begin{example}
    Take $T=\begin{bmatrix}11&19&25&26\end{bmatrix}$. Then, using the \textit{mbfailures} command in \cite{githubM2DistanceReduction},
    as demonstrated in Example \ref{ex: same bouquet diff DR}, we see that $T$ has a unique minimal Markov basis $M$ that distance-reduces every element of $\operatorname{Gr}(T)$ apart from $z=\begin{pmatrix}-1&2&1&-2\end{pmatrix}$. Furthermore, $z$ is not a circuit. However, taking any homogeneous $T_{[4]}$-robust matrix $A$ (this choice does not matter here since, by Theorem \ref{thm:globaldr}, all homogeneous $T_{[4]}$-robust matrices will have the same distance-reducing properties), then $A$ also has a unique minimal Markov basis (since the minimal Markov bases of $T$ and $A$ are in bijection by \cite[Theorem 2.2]{kosta2023strongly}) which distance-reduces all elements of $\operatorname{Gr}(A)$ including $D(z)$. 

    We note that $\hat z=\begin{pmatrix}-4&1&1&0\end{pmatrix} \in M $ where $\hat z^+ \le z^+$ and $\operatorname{Supp}(\hat z^-) \cap \operatorname{Supp}(z^-)=\{1\}$. Using these properties, we can infer that $D(\hat z)^+ \le D(z)^+$ and $\operatorname{Supp}(D(\hat z)^-) \cap \operatorname{Supp}(D(z)^-) \neq \emptyset$ so, by Proposition \ref{prop:drredef}, $D(\hat z)$ distance-reduces $D(z)$. However, since $T$ is not homogeneous we can't use the fact that $\hat z^+ \le z^+$ and $\operatorname{Supp}(\hat z^-) \cap \operatorname{Supp}(z^-)=\{1\}$ to imply that $\hat z$ distance-reduces $z$ since Proposition \ref{prop:drredef} does not apply. In fact, $\hat z$ doesn't distance-reduce $z$ due to the fact that $\lVert z-\hat z\rVert=6=\lVert z \rVert$ so condition (\ref{item:drbouquet2}) of Definition \ref{def:drbouquet} is not satisfied.
\end{example}

\section*{Acknowledgements}
D. Kosta gratefully acknowledges funding from the Royal Society Dorothy Hodgkin Research Fellowship DHF$\backslash$R1$\backslash$201246 and DHF$\backslash$R$\backslash$251007.

\bibliographystyle{plain}
\bibliography{references}

\appendix

\section{Proof of Proposition \ref{prop:3graphs}}

Let $a_1,\dots ,a_n$ denote the columns of our matrix $A$ and let $e_i$ denote the $i$th unit vector.

\begin{definition}[\cite{CKT_2007}, Section 2] \label{def:partord}
    We define a partial ordering on $\mathbb{Z}^d$, where, given $b,c \in \mathbb{Z}^d$, we say $b \sqsubseteq c$ if $c-b \in \NN A$. If $b \sqsubseteq c$ and $b \neq c$ then we write $b \sqsubset c$.
\end{definition}

Using this partial ordering, we can now rewrite the ideal $I_{A,b}$ from Definition \ref{def:fiberGraph} as 
$$I_{A,b} = \langle x^u-x^v \,|\,  \deg_A(x^u)=\deg_A(x^v) \sqsubset b \rangle.$$
Recall that we then construct the graph $G_b$ with vertex set $\mathcal F_b$ and edge set 
$E(G_b) = \{\{u,v\} \,|\, x^u-x^v \in I_{A,b}\}$.
We also recall Definition \ref{def:3difgra} where we define the graph $G'_b$ also on the vertex set $\mathcal{F}_b$ whose edges are $E(G'_b)=\{\{u,v\} \,|\, \operatorname{Supp}(u) \cap \operatorname{Supp}(v) \neq \emptyset\}$.

\begin{lemma}[\cite{CKT_2007}, Proposition 2.2] \label{lem:CKTprop2.2}
     Each connected component of $G_b$ is a complete subgraph.
\end{lemma}

No proof of Lemma \ref{lem:CKTprop2.2} is provided in \cite{CKT_2007} so we provide a short proof here.

\begin{proof}
    Suppose $\{u,v\} \in E(G_b)$ and $\{v,w\} \in E(G_b)$ for some $u,v,w \in \mathcal{F}_b$, then $x^{u}-x^{v},x^{v}-x^{w} \in I_{A,b}$. Adding them together, we have $x^{u}-x^{w} \in I_{A,b}$ and then $\{u,w\} \in E(G_b)$. Thus, if $u$ and $v$ lie in the same connected component of $G_b$, there is a path $u=w_0,\dots,w_p=v$ where $\{w_{l-1},w_l\} \in E(G_b)$ for all $l \in [p]$. Thus, by our earlier observation and by induction, $\{w_0,w_l\} \in E(G_b)$ for all $l \in[p]$ and so, in particular, $\{u=w_0,v=w_p\} \in E(G_b)$.
\end{proof}

\begin{lemma} \label{lem:edgeinclusion}
    We have the inclusion $E(G'_b) \subseteq E(G_b)$.
\end{lemma}

\begin{proof}
    Take $u,v \in \mathcal{F}_b$ such that $\{u,v\} \in E(G'_b)$. Thus, $\operatorname{Supp}(u) \cap \operatorname{Supp}(v)$ is non-empty and we can find some $j \in \operatorname{Supp}(u) \cap \operatorname{Supp}(v)$. Now define $u':=u-e_j$ and $v':=v-e_j$ where, by the way we found $j$, the vectors $u',v' \in \NN^n$ still. Also note that $\deg_A(x^{u'})=Au'=Au-Ae_j=b-a_j$ and, similarly, $\deg_A(x^{v'})=b-a_j$. If $a_j=0$ then $Ae_j=0$ so $e_j \in \ker_\ZZ(A) \cap \NN^n$ which is a contradiction to our assumption that $ \ker_\ZZ(A) \cap \NN^n=\{0\}$. Thus, $a_j \neq 0$, so $b-a_j \neq b$. In addition, $b-(b-a_j)=a_j = Ae_j \in \NN A$ so $b-a_j \sqsubset b$. Putting this together we have $\deg_A(x^{u'})=\deg_A(x^{v'}) \sqsubset b$ and thus, $x^{u'}-x^{v'} \in I_{A,b}$. Finally, $x^{u}-x^{v}=x_j(x^{u'}-x^{v'}) \in I_{A,b}$ so $\{u,v\} \in E(G_b)$.
\end{proof}

\begin{proposition} \label{prop:samegraphccs}
    Graphs $G_b$ and $G'_b$ have the same connected components.
\end{proposition}

\begin{proof}
    We first show that, if $u,v$ are in the same connected component of $G_b$, then $u,v$ are in the same connected component of $G'_b$. By Lemma \ref{lem:CKTprop2.2}, if we take $u,v \in \mathcal{F}_b$ such that $u,v$ are in the same connected component of $G_b$ then $\{u,v\} \in E(G_b)$. Then, $x^u-x^v \in I_{A,b}$ so we can write $x^u-x^v=\sum_{i=1}^p x^{c_i}(x^{u_i}-x^{v_i})$ for some $p \in \NN$ where $c_i,u_i,v_i \in \NN^n$ and $\deg_A(x^{u_i})=\deg_A(x^{v_i})=b_i \sqsubset b$ for $i \in [p]$. Furthermore, we can specify that $x^u=x^{c_1} x^{u_1}$, $x^v=x^{c_p} x^{v_p}$ and $x^{c_i}x^{v_i}=x^{c_{i+1}}x^{u_{i+1}}$ for $i \in [p-1]$. Thus, for any $i \in [p]$, $\deg_A(x^{c_i}x^{u_i})=\deg_A(x^{c_i}x^{v_i})=b$ and, since $b_i \neq b$, then $c_i \neq 0$. We can now define a path $w_i \in \mathcal{F}_b$ for $0 \le i \le p$ where $w_0=c_1+u_1$ and  $w_i=c_i+v_i$ for $i \in [p]$. In particular, $w_0=u$ and $w_p=v$ so it is a path from $u$ to $v$. Also, $c_i \neq 0$ so we can take $j_i \in \operatorname{Supp}(c_i)$ for all $i \in [p]$. Since $c_i \le c_i+v_i=w_i$ and $c_i \le c_i+u_i=c_{i-1}+v_{i-1}=w_{i-1}$ then $j_i \in \operatorname{Supp}(w_i) \cap \operatorname{Supp}(w_{i-1})$. Thus, $\{w_i,w_{i-1}\} \in E(G'_b)$ for all $i \in [p]$ and $u=w_0,\dots,w_p=v$ is a path in $G'_b$. Thus, $u,v$ lie in the same connected component of $G'_b$.

    On the other hand, by Lemma \ref{lem:edgeinclusion}, if $u,v$ are in the same connected component of $G'_b$ then they are in the same connected component of $G_b$ as well. Thus we have showed that the connected components of both graphs are the same.
\end{proof}



\end{document}